\documentclass[11pt]{article}
\usepackage{titling}

\usepackage[T1]{fontenc}
\usepackage[hmargin=2.5cm,vmargin=2.5cm]{geometry}
\usepackage{amsmath,amsthm,amssymb,mathtools}
\usepackage{thm-restate,thmtools}
\usepackage{xcolor}
\colorlet{sandbox}{orange!4!white}
\colorlet{sandborder}{orange!40!gray}
\colorlet{bluebox}{blue!4!white}
\colorlet{blueborder}{blue!20!gray}
\usepackage{bbm}
\usepackage[numbers]{natbib}
\usepackage{graphicx}
\usepackage{yfonts}
\usepackage{lineno}
\usepackage[normalem]{ulem}
\usepackage{todonotes}
\usepackage[pdftex,colorlinks,linkcolor=black,urlcolor=black,citecolor=black, hypertexnames=false]{hyperref}
\hypersetup{
colorlinks=true,
    linkcolor={red!50!black},
    citecolor={blue!50!black},
    urlcolor={blue!80!black},
bookmarksopen=true,
bookmarksnumbered,
bookmarksopenlevel=2,
bookmarksdepth=3
}
\usepackage[capitalise,nameinlink]{cleveref}
\usepackage{booktabs}
\usepackage{mathrsfs}
\usepackage{tabularx,tablefootnote}
\usepackage{microtype}
\newtheorem{theorem}{Theorem}[section]
\newtheorem{lemma}[theorem]{Lemma}
\newtheorem{corollary}[theorem]{Corollary}
\newtheorem{proposition}[theorem]{Proposition}
\newtheorem{conjecture}[theorem]{Conjecture}

\theoremstyle{definition}
\newtheorem{observation}[theorem]{Observation}

\newtheorem{remark}[theorem]{Remark}

\newcounter{tbox}

\makeatletter
\newcommand{\leqnomode}{\tagsleft@true}
\newcommand{\reqnomode}{\tagsleft@false}
\makeatother

\newcommand*{\myproofname}{Proof}
\newenvironment{claimproof}[1][\myproofname]{\begin{proof}[#1]}{\end{proof}}

\newtheoremstyle{parens}
  {}
  {}
  {\itshape}
  {\parindent}
  {}
  {}
  {.5em}
  {(\thmnumber{#2})\thmnote{ [#3]}}

\theoremstyle{parens}
\newtheorem{nitem}{}[section]

\crefname{theorem}{Thm.}{Thms.}
\Crefname{theorem}{Theorem}{Theorems}
\crefname{lemma}{Lem.}{Lems.}
\Crefname{lemma}{Lemma}{Lemmas}
\crefname{corollary}{Cor.}{Cors.}
\Crefname{corollary}{Corollary}{Corollaries}
\crefname{proposition}{Prop.}{Props.}
\Crefname{proposition}{Proposition}{Propositions}
\crefname{conjecture}{Conj.}{Conjs.}
\Crefname{conjecture}{Conjecture}{Conjectures}

\newcommand{\tw}{\mathsf{tw}}
\newcommand{\tin}{\mathsf{tree}\textnormal{-}\alpha}
\newcommand{\dg}{\mathsf{dgn}}
\newcommand{\ad}{\mathsf{ad}}
\newcommand{\N}{\mathbb{N}}

\newcommand{\sepalpha}{\mathsf{sep}\text{-}\alpha}
\newcommand{\ltw}{\mathsf{ltw}} \newcommand{\ltreealpha}{\mathsf{ltree}\textnormal{-}\alpha}
\newcommand{\sepnum}{\mathsf{sn}}

\newcommand{\omegar}{\omega_r}
\newcommand{\C}{\mathcal C}
\newcommand{\girth}{\mathsf{girth}}
\newcommand{\diam}{\mathsf{diam}}

\usepackage{tikz}
\usepackage{tcolorbox}
\usepackage{enumitem}
\usetikzlibrary{arrows}
\usetikzlibrary{arrows.meta}
\usetikzlibrary{decorations.pathmorphing, decorations.pathreplacing, decorations.shapes}
\usetikzlibrary{decorations.markings}
\usetikzlibrary{calc}
\usetikzlibrary{shapes.geometric}

\pgfkeys{/tikz/.cd,
    contour distance/.store in=\ContourDistance,
    contour distance=0pt, 
    contour step/.store in=\ContourStep,
    contour step=1pt,
}

\pgfdeclaredecoration{closed contour}{initial}
{%
\state{initial}[width=\ContourStep,next state=cont] {
    \pgfmoveto{\pgfpoint{\ContourStep}{\ContourDistance}}
    \pgfcoordinate{first}{\pgfpoint{\ContourStep}{\ContourDistance}}
    \pgfpathlineto{\pgfpoint{0.3\pgflinewidth}{\ContourDistance}}
    \pgfcoordinate{lastup}{\pgfpoint{1pt}{\ContourDistance}}
    
  }
  \state{cont}[width=\ContourStep]{
     \pgfmoveto{\pgfpointanchor{lastup}{center}}
     \pgfpathlineto{\pgfpoint{\ContourStep}{\ContourDistance}}
     \pgfcoordinate{lastup}{\pgfpoint{\ContourStep}{\ContourDistance}}
  }
  \state{final}[width=\ContourStep]
  { 
    \pgfmoveto{\pgfpointanchor{lastup}{center}}
    \pgfpathlineto{\pgfpointanchor{first}{center}}
  }
}

\tikzset{
  cir/.style = {circle,draw,fill,inner sep=.7pt},
  circ/.style = {circle,draw,fill,inner sep=1.3pt},
  circg/.style = {circle,draw=lightgray,fill=lightgray,inner sep=1.3pt},
  circr/.style = {circle,draw=Crimson,fill=Crimson,inner sep=1.3pt},
  invisible/.style = {circle,draw=none,inner sep=0pt,font=\tiny},
  nonedge/.style={decorate,decoration={snake,amplitude=.3mm,segment length=1mm},draw},
}

\tcbuselibrary{skins}

\newtcolorbox{mybox}[1]{minipage boxed title*=-2cm,
enhanced,attach boxed title to top center=
{yshift=-3mm,yshifttext=-1mm},colback=Lavender!30!white,
boxed title style={size=small,colback=Lavender},coltitle=black,
center title,title={#1}}

\newcommand{\email}[1]{%
  \texttt{#1}%
}

\newcommand\extrafootertext[1]{%
    \bgroup
    \renewcommand\thefootnote{\fnsymbol{footnote}}%
    \renewcommand\thempfootnote{\fnsymbol{mpfootnote}}%
    \footnotetext[0]{#1}%
    \egroup
}
\definecolor{sagegreen}{RGB}{143,161,125}
\usepackage{authblk}

\title{Clique-dependent strongly sublinear treewidth and strongly sublinear tree-independence number}
\date{}

\author{Andrea Munaro}

\affil{Department of Mathematical, Physical and Computer Sciences, University of Parma, Italy}

\begin{document}

\maketitle
\extrafootertext{\scriptsize Email: \email{andrea.munaro@unipr.it}.}


\begin{abstract} We establish a strongly sublinear counterpart of a recent result of Chudnovsky, E S, and Lokshtanov (arXiv 2025) on treewidth and tree-independence number. Namely, we prove that a hereditary graph class has strongly sublinear tree-independence number if and only if, for every fixed clique bound, its graphs of bounded clique number have strongly sublinear treewidth. In fact, this is part of a broader equivalence theorem. For hereditary classes, these conditions are also equivalent to having clique-dependent polynomial expansion, to admitting balanced separators whose size is bounded by $K\omega(G)^s |V(G)|^{1-\beta}$ for fixed $K,s,\beta>0$, and to admitting balanced clique-based separators of strongly sublinear size (equivalently, weight). Thus, we show that all these properties, which arose independently in the study of subexponential-time exact algorithms and polynomial-time approximation schemes, in fact describe the same hereditary graph classes. As a consequence of our equivalence theorem, we also show that every hereditary class $\mathcal C$ with strongly sublinear tree-independence number admits a subexponential-time algorithm that, given $G\in\mathcal C$, computes a tree decomposition of $G$ with strongly sublinear independence number.

\end{abstract}

\section{Introduction}

Tree decompositions are a fundamental tool in structural and algorithmic graph
theory. The classical measure of the complexity of a tree decomposition is given by the maximum size of its bags, leading to the notion of treewidth $\tw(G)$ of a graph $G$. However, since large cliques alone force large treewidth, it is natural to look for new measures of complexity that are insensitive to clique size. One particularly successful width parameter arising in this way is the tree-independence number $\tin(G)$ of a graph $G$, introduced independently by \citet{Yolov18}
and \citet{DMS24}. Here, instead
of bounding the size of every bag, one asks that the independence number of
every bag be bounded. Similarly to treewidth, tree-independence
number has important algorithmic consequences; in particular, several optimization problems that are hard on general graphs become tractable on
classes of bounded tree-independence number (see, e.g., \cite{DMS24,LMMORS24,LPR26,Yolov18}).

A natural question is then how bounds on treewidth in terms of the clique number $\omega$ translate into bounds on tree-independence number and vice versa. A simple application of Ramsey’s theorem shows that every hereditary graph class $\mathcal{C}$ of bounded (by a constant) tree-independence number is
$(\tw, \omega)$-bounded, i.e., there exists a function $f$ such that $\tw(G) \leq f(\omega(G))$ for every $G \in \mathcal{C}$, and a well-known conjecture of \citet[Conjecture~8.5]{DMS24a} asserts that the converse holds as well, namely every hereditary
$(\tw,\omega)$-bounded class has bounded (by a constant) tree-independence number. This conjecture was recently disproved by \citet{CT25}, even in the case the binding function $f$ is polynomial. However, \citet{CESL25} showed that a quantitative
version of the conjecture survives at the polylogarithmic scale.
More precisely, they proved that, for a hereditary class $\mathcal C$, having $\tin(G)\le (\log |V(G)|)^{O(1)}$ for every $G\in\mathcal C$ on at least three vertices is equivalent to having $\tw(G)\le (\omega(G)\log |V(G)|)^{O(1)}$.

The starting point of this paper is the corresponding question at the \emph{strongly sublinear} scale. We prove that a similar behavior occurs: For a hereditary class, having strongly sublinear treewidth on every bounded-clique subclass is equivalent to having strongly sublinear tree-independence number. To formalize this, we introduce the following terminology. A hereditary class $\C$ is \emph{strongly sublinearly
$(\tw,\omega)$-bounded} if there exist functions
$f\colon\mathbb N\to\mathbb R_{>0}$ and
$g\colon\mathbb N\to(0,1]$ such that every $G\in\C$ satisfies $\tw(G)
  \le f(\omega(G))|V(G)|^{1-g(\omega(G))}$. Note that no restriction is imposed on the dependence of either function on the clique number. We show that a hereditary class $\mathcal{C}$ is strongly sublinearly
$(\tw,\omega)$-bounded if and only if it has strongly sublinear tree-independence number, i.e., there exist constants $c\geq 1$ and $\varepsilon > 0$ such that $\tin(G) \leq c|V(G)|^{1-\varepsilon}$, for every $G \in \mathcal{C}$. 

Our result turns out to be considerably more general. In particular, it identifies the preceding properties with several other properties that arose independently in the study of subexponential-time exact algorithms and polynomial-time approximation schemes, especially in relation to geometric intersection graphs. In order to explain this, we need to recall some definitions.

Let $G$ be a graph. A set $S\subseteq V(G)$ is a \emph{balanced separator} of $G$ if every
connected component of $G-S$ has at most $2|V(G)|/3$ vertices. Let $\mathsf{sep}(G):= \min\{|S| : S\subseteq V(G)\text{ is a balanced separator of }G\}$, and let $\sepalpha(G):= \min\{\alpha(G[S]): S\subseteq V(G)\text{ is a balanced separator of }G\}$. A graph class $\mathcal{C}$ has
\emph{strongly sublinear separators} if there exist constants $c\ge1$ and
$\varepsilon>0$ such that $\mathsf{sep}(G)\le c|V(G)|^{1-\varepsilon}$, for every $G \in \mathcal{C}$, and \emph{strongly sublinear $\alpha$-separators} if
there exist constants $c\ge1$ and $\varepsilon>0$ such that $\sepalpha(G)\le c|V(G)|^{1-\varepsilon}$, for every $G \in \mathcal{C}$. A graph class $\mathcal C$ has \emph{sublinear separators} if there exists a function $f(n)=o(n)$ such that every $n$-vertex graph $G\in\mathcal C$ has a balanced separator of size at most $f(n)$. Finally, we say that
a hereditary class $\C$ is \emph{strongly sublinearly
$(\mathsf{sep},\omega)$-bounded} if there exist functions
$f\colon\mathbb N\to\mathbb R_{>0}$ and
$g\colon\mathbb N\to(0,1]$ such that every $G\in\C$ satisfies $\mathsf{sep}(G)
\le f(\omega(G))|V(G)|^{1-g(\omega(G))}$. It follows from \cite{DN19} (see also \Cref{DNsep}) that a hereditary class is strongly sublinearly
$(\mathsf{sep},\omega)$-bounded if and only if it is strongly sublinearly
$(\tw,\omega)$-bounded.  

For a graph class $\C$ and $q\in\mathbb{N}$, let $\C_{\le q}:=\{G\in\C:\omega(G)\le q\}$. Following \citet{DLPSXZ23}, we say that a class $\C$ has \emph{clique-dependent polynomial
expansion} if every subclass of $\C$ of bounded clique number has polynomial
expansion; equivalently, every slice $\C_{\le q}$ has polynomial expansion.
For hereditary classes, a well-known theorem of \citet{DN16} implies that this is equivalent to having strongly sublinear separators on every bounded-clique slice. 

A second property is related to clique-based separators \cite{ABT25,DMMY25,BBK20,dBKMT23}. A \emph{clique-based separator} of a graph $G$ is a collection $\mathcal{S}$ of vertex-disjoint cliques whose union is a balanced 
separator of $G$. The \textit{size} of $\mathcal{S}$ is $|\mathcal{S}|$ and the \textit{weight} of $\mathcal{S}$ is the quantity $w(\mathcal{S}) := \sum_{C\in\mathcal{S}} \log(|C|+1)$. We say that a class $\mathcal{C}$ has \emph{clique-based separators of strongly sublinear size} (resp., \emph{weight}) if there exists $\varepsilon>0$ such that every $n$-vertex graph in $\mathcal{C}$ has a clique-based separator of size (resp., weight) $O(n^{1-\varepsilon})$, and $\mathcal C$ has \emph{sublinear-weight clique-based separators} if there exists a function $f(n)=o(n)$ such that every $n$-vertex graph $G\in\mathcal C$ has a clique-based separator of weight at most $f(n)$. 

A third property, considered by \citet{LPSXZ25}, is the following. For $K,s\ge0$, and $\beta>0$, a graph $G$ admits \emph{balanced $(K, s, \beta)$-separators}\footnote{In their definition, \citet{LPSXZ25} use $1/2$-balanced separators. This is equivalent to our formulation up to a constant-factor change.} if every induced subgraph $H$ of $G$ satisfies $\mathsf{sep}(H)
\le K\omega(H)^s |V(H)|^{1-\beta}$. Note that, by \cite{DN16}, if a graph class admits balanced $(K, s, \beta)$-separators, then the class has clique-dependent polynomial expansion.  

Our main theorem shows that all these properties describe exactly the same
hereditary graph classes (see also \Cref{reldiagram}).

\begin{restatable}{theorem}{equivalence}
\label{equivalence}
Let $\C$ be a hereditary class. The following are equivalent.
\begin{enumerate}[label=(\roman*)]
  \item $\C$ has clique-dependent polynomial expansion;
  \item $\C$ is strongly sublinearly $(\mathsf{sep},\omega)$-bounded;
  \item $\C$ is strongly sublinearly $(\tw,\omega)$-bounded;
  \item there exist constants $K\ge1$, $\beta>0$, and $s\ge0$ such that every $G\in\C$ satisfies $\mathsf{sep}(G)
   \le K\,\omega(G)^s |V(G)|^{1-\beta}$;
  \item $\C$ has clique-based separators of strongly sublinear size;
  \item $\C$ has clique-based separators of strongly sublinear weight;
  \item $\C$ has strongly sublinear tree-independence number;
  \item $\C$ has strongly sublinear $\alpha$-separators.
\end{enumerate}
\end{restatable}

In particular, the equivalence \textup{(iii)}$\Leftrightarrow$\textup{(vii)} is the
strongly sublinear counterpart of the polylogarithmic theorem of \citet{CESL25} mentioned above. There is, however, an additional phenomenon here.
Condition~\textup{(iii)} permits a different power saving for every clique bound, whereas condition~\textup{(vii)} has a single exponent valid on
the whole class. Thus, \Cref{equivalence} contains a nontrivial
\emph{uniformization} statement. 
Indeed, \textup{(ii)}$\Rightarrow$\textup{(iv)} shows that if every
bounded-clique slice has strongly sublinear separators, then there exist
constants $K\ge1$, $\beta>0$, and $s\ge0$ such that
$\mathsf{sep}(G)
\le K\omega(G)^s |V(G)|^{1-\beta}$, for every $G\in\C$. We prove this as \Cref{poly-coefficient}. Thus, both sources of nonuniformity in Conditions~\textup{(ii)} and ~\textup{(iii)} can be controlled:
the power saving in $|V(G)|$ becomes independent of the clique number, and
the remaining multiplicative dependence on the clique number can be chosen
polynomial. 

Note that most implications in \Cref{equivalence} either follow from definitions (e.g., \textup{(v)}$\Leftrightarrow$\textup{(vi)}), or from known results (e.g., \textup{(ii)}$\Leftrightarrow$\textup{(iii)}), and the main work is devoted to showing \textup{(ii)}$\Rightarrow$\textup{(iv)}$\Rightarrow$\textup{(v)}. The proof of this spans \Cref{sec:sparsification,sec:uniformization,sec:theproof} and we provide a sketch in \Cref{overview}.  

We also remark that the full set of equivalences in \Cref{equivalence} for the strongly sublinear regime ceases to hold for the polylogarithmic regime. For example, there are hereditary classes with bounded (by a constant) tree-independence number and no clique-based separators of polylogarithmic size (see \Cref{polylog}).

\Cref{equivalence} has several interesting consequences, particularly from the viewpoint of geometric intersection graphs, as we explain next. Even though geometric intersection graphs need not be sparse, after bounding the clique number many important such classes behave like sparse classes. \citet{DLPSXZ23} showed that pseudo-disk graphs have clique-dependent polynomial expansion and it is known that the same holds for intersection graphs of convex fat objects (see, e.g., \cite{HP17,MTTV97,SW98}). These properties were used in \cite{DLPSXZ23} to obtain
approximation algorithms for \textsc{Subgraph Hitting} and related problems. \citet{LPSXZ25} observed that pseudo-disk graphs and intersection graphs of convex fat objects not only have clique-dependent polynomial expansion but in fact admit balanced $(K, s, \beta)$-separators. Moreover, they used this property to develop a subexponential parameterized
framework for \textsc{Weighted $\mathcal F$-Hitting}. Similarly, clique-based separators have played a central role in subexponential-time algorithms. \citet{BBK20} showed that, for every fixed dimension $d \geq 2$, intersection graphs of convex fat objects in
$\mathbb R^d$ admit clique-based separators of weight $O(n^{1-1/d})$, while \citet{dBKMT23} proved a bound of
$O(n^{2/3}\log n)$ for pseudo-disk graphs, together with analogous results
for several other geometric intersection classes. \Cref{equivalence} shows that the appearance of clique-dependent polynomial expansion, balanced $(K, s, \beta)$-separators, and clique-based separators of strongly sublinear weight is not coincidental: for hereditary classes, they are three
manifestations of the same structural property. 

Another important consequence of \Cref{equivalence} is related to the fact that the different equivalent conditions come with
complementary algorithmic frameworks, as already suggested by the previous paragraph. Indeed, clique-dependent polynomial expansion
was introduced in the study of approximation algorithms for \textsc{Subgraph
Hitting}~\cite{DLPSXZ23}, whereas classes admitting balanced $(K, s, \beta)$-separators are the setting of the subexponential parameterized
framework for \textsc{Weighted $\mathcal F$-Hitting} from \citet{LPSXZ25}. Moreover, clique-based separators of sublinear weight yield subexponential-time algorithms for problems such as \textsc{Independent Set}, \textsc{Feedback Vertex Set}, and \textsc{$q$-Coloring} \cite{BBK20,dBKMT23}. Finally, \citet{LPR26} showed that a broad family of problems asking for a max-weight induced subgraph of bounded treewidth satisfying a given $\mathsf{CMSO}_2$ property can be solved in time $n^{O(k)}$ on graphs of tree-independence number $k$, provided a tree decomposition of independence
number at most $k$ and $n^{O(1)}$ nodes is given (see also \cite{LMMORS24,Yolov18}). Their conditional statement is relevant when $k$ is strongly sublinear. Indeed, the general algorithm of \citet{DFGKM26} constructs, on graphs of tree-independence number at most $k$, a tree decomposition of independence number at most $8k$ in time $2^{O(k^2)}n^{O(k)}$; as observed in \cite{LPR26}, the factor $2^{O(k^2)}$ prevents one from directly obtaining subexponential-time algorithms when $k=\Omega(\sqrt n)$. We show that this preprocessing obstacle can be overcome for every hereditary class satisfying the equivalent conditions of \Cref{equivalence} by reducing to the single-exponential treewidth approximation algorithm of \citet{Kor23} (see \Cref{overview} for more details).

\begin{restatable}{corollary}{algo}
\label{algo}
Let $\mathcal C$ be a hereditary class with strongly sublinear
tree-independence number. Then there exist constants
$c\ge1$ and $\varepsilon>0$ and an algorithm
that, given an $n$-vertex graph $G\in\mathcal C$, constructs
in time $2^{O(n^{1-\varepsilon})}$ a tree decomposition of
$G$ with at most $n$ bags and independence number at most
$cn^{1-\varepsilon}$.
\end{restatable}

Hence, every hereditary class satisfying any of the equivalent conditions in \Cref{equivalence} enjoys the algorithmic consequences of these frameworks. 

In this paper, we also investigate how the picture given by \Cref{equivalence} interacts with another algorithmically useful notion, namely fractional $\tin$-fragility. Let us first recall some definitions. Let $p$ be a width parameter in $\{\tw, \tin\}$. For $0\leq \beta \leq 1$, a \textit{$\beta$-general cover} of a graph $G$ is a multiset $\mathcal{C}$ of subsets of $V(G)$ such that each vertex belongs to at least $\beta|\mathcal{C}|$ elements of the cover. The \textit{$p$-width} of the cover is $\max_{C \in \mathcal{C}}p(G[C])$. For a parameter $p$, a graph class $\mathcal{G}$ is \textit{fractionally $p$-fragile} if there exists a \emph{fragility function} $f\colon\mathbb{N}\rightarrow\mathbb{N}$ such that, for every $r\in\mathbb{N}$, every $G \in \mathcal{G}$ has a $(1 - 1/r)$-general cover with $p$-width at most $f(r)$. A fractionally $p$-fragile class $\mathcal{G}$ is \textit{efficiently fractionally $p$-fragile} if there exists an algorithm that, for every $r\in\mathbb{N}$ and $G \in \mathcal{G}$, returns in $\mathsf{poly}(|V(G)|)$ time a $(1 - 1/r)$-general cover $\mathcal{C}$ of $G$ and, for each $C \in \mathcal{C}$, a tree decomposition of $G[C]$ of width (if $p = \tw$) or independence number (if $p = \tin$) at most $f(r)$, for some function $f\colon\mathbb{N}\rightarrow\mathbb{N}$. Fractional $\tw$-fragility was introduced by \citet{Dvo16} and the definition above is taken from \cite{GMY23}. Efficient fractional $\tin$-fragility allows for PTASes for a large number of maximization problems whose task is finding a max-weight induced subgraph with bounded clique number satisfying a fixed $\mathsf{CMSO_2}$ formula expressing a near-monotone property (see \cite{GMY24} for details).

In \Cref{sec:frac-tree-alpha-separators}, building on \cite{Dvo16}, we observe that $\sepalpha(G)\le |V(G)|/r+f(r)$,
for every graph $G$ in a fractionally $\tin$-fragile class with fragility function $f$. As a first consequence, we show that every fractionally $\tin$-fragile class has sublinear-weight clique-based separators. This answers in the positive a question of \citet{GMY24} asking whether every graph class of bounded layered tree-independence
number has sublinear-weight clique-based separators.
A second consequence is that polynomial-rate fractional $\tin$-fragility (i.e., the fragility function is polynomial) implies the
equivalent properties of \Cref{equivalence}. We conjecture the following fractional counterpart of our main equivalence.

\begin{conjecture}\label{alphaDvorak}  Every hereditary class admitting strongly sublinear $\alpha$-separators is fractionally $\tin$-fragile.
\end{conjecture}  

Note that the fragility function in \Cref{alphaDvorak} might be polynomial; we are not aware of any counterexample to this. \Cref{alphaDvorak} is an
$\alpha$-analogue of the following conjecture of \citet{Dvo16} and we next observe that it would imply it.

\begin{conjecture}[\citet{Dvo16}]\label{Dvorak} Every hereditary class admitting strongly sublinear separators is fractionally $\tw$-fragile.
\end{conjecture}

\begin{lemma}
\label{alpha-dvorak-implies-dvorak}
If \Cref{alphaDvorak} holds, then \Cref{Dvorak} holds.
\end{lemma}

\begin{proof}
Let $\mathcal C$ be a hereditary class admitting strongly sublinear separators. In particular, $\mathcal C$ admits strongly sublinear $\alpha$-separators. Assume now that \Cref{alphaDvorak} holds. Then $\mathcal C$ is fractionally $\tin$-fragile, for some fragility function $f$. Moreover, there exists an integer $d$ such that every graph in $\mathcal C$ is $d$-degenerate \cite{DN16} (see also \Cref{degeneracy-from-sep}). 

Fix now $r\in\mathbb N$. For every $G\in\mathcal C$, let $\mathcal A$ be a $(1-1/r)$-general cover of $G$ of $\tin$-width at most $f(r)$. For each $A\in\mathcal A$, the graph $G[A]$ is $d$-degenerate and hence $(d+1)$-colorable. Hence every bag $B$ of a tree decomposition of $G[A]$ satisfies $|B|\le (d+1)\alpha(G[B])$. Choosing a tree decomposition of $G[A]$ with independence number at most $f(r)$ gives $
\tw(G[A])\le (d+1)f(r)-1$. Thus the same cover $\mathcal A$ has $\tw$-width at most $(d+1)f(r)-1$, and $\mathcal C$ is fractionally $\tw$-fragile.
\end{proof}

Another indication that \Cref{alphaDvorak} might be challenging is that an efficient version of it, together with results from \cite{dBKMT23,GMY24}, would immediately imply that \textsc{Max-Weight Independent Set} admits a PTAS on pseudo-disk graphs. Finding such a PTAS is a well-known open problem \cite{CHP12,CAPP25}.

In \Cref{sec:layered-tree-alpha}, we consider a natural question suggested by \cite[Conjecture~8.5]{DMS24a} and \Cref{equivalence} (see also \Cref{reldiagram}): If every bounded-clique slice of a hereditary class has bounded layered treewidth, does the whole class have bounded layered tree-independence number? We show that an easy adaptation of the constructions from \cite{CT25} provides a negative answer. Hence, a ``layered variant'' of \cite[Conjecture~8.5]{DMS24a} cannot hold either. 

Finally, in \Cref{sec:counter}, we show that there exist monotone classes admitting sublinear separators which are not $\chi$-bounded. We obtain this by modifying a construction of \citet{MS18} of graphs with arbitrarily large average degree and girth and whose $t$-vertex
subgraphs have balanced separators of size $o(t)$.  
This answers in a strong form a question of \citet{DMMY25} asking whether every graph class admitting sublinear-weight clique-based separators has bounded $\alpha$-degeneracy. 

\subsection{Overview of the proofs of \Cref{equivalence} and \Cref{algo}}\label{overview}
We first consider \Cref{equivalence}. Here, we only sketch the proofs of the main implications \textup{(ii)}$\Rightarrow$\textup{(iv)}$\Rightarrow$\textup{(v)}. Then we pass to \Cref{algo}.

Suppose that every bounded-clique slice of the hereditary class $\C$ has strongly sublinear separators. In \textup{(ii)}$\Rightarrow$\textup{(iv)}, we need to replace the a priori arbitrary dependence of both the coefficient and the exponent on the clique number by a bound of the form
\[
   \mathsf{sep}(G) \le K\omega(G)^s |V(G)|^{1-\beta}, \tag{$\ast$}
\]
where $K,s$ and $\beta>0$ are independent of $G$. This uniformization result is the technically difficult part of the proof of \Cref{equivalence} and is the content of \Cref{poly-coefficient}. We proceed as follows.

Strongly sublinear separators on each fixed-clique slice imply bounded degeneracy on the slice, and hence $\C$ is $(\dg,\omega)$-bounded. By the polynomialization theorem for $(\dg,\omega)$-bounded classes \cite{DMcC25,GH25}, the degeneracy can therefore be bounded by a fixed polynomial in $\omega$. 

We then prove an induced sparsification lemma (\Cref{induced-sparsification}). Roughly speaking, if a $d$-degenerate graph contains a sufficiently large $r$-shallow clique minor, then it contains an induced $K_7$-free subgraph of controlled order and large treewidth. The proof of \Cref{induced-sparsification} organizes many branch sets of the clique model into blocks indexed by the vertices of a cubic expander and, for every expander edge, fixes connections between all possible pairs of branch sets in the corresponding blocks. We can therefore select one branch set from each block independently at random while preserving the expander as a minor. We then prove that the at most three attachment points arising in each selected branch set can be joined by a small induced connector of clique number at most three. It follows that any $K_7$ in the resulting induced subgraph must meet at least three independently selected branch sets, and hence survives the random choice only with small probability. Finally, bounded degeneracy gives a sufficiently strong bound on the number of potential $K_7$'s, allowing us to eliminate all surviving ones while preserving large treewidth.

Since $\C$ is hereditary, the resulting $K_7$-free graph belongs to $\C_{\le6}$.  The strongly sublinear separator bound on this single slice therefore gives an upper bound on its treewidth, and comparing the two treewidth bounds yields $\omega_r(G)\le p(\omega(G))(r+1)^\delta$
for a fixed polynomial $p$ and a fixed exponent $\delta$ (see \Cref{shallow-clique}). The Plotkin--Rao--Smith separator theorem \cite{PRS94} (see \Cref{PRS}) then converts this polynomial shallow-clique bound into \((\ast)\). Note that the crucial outcome is not just that $\beta>0$ can be chosen uniformly in the clique number, but also that the remaining dependence on the clique number can be made polynomial.

In \textup{(iv)}$\Rightarrow$\textup{(v)}, we use \((\ast)\) to obtain a clique-based separator of strongly sublinear size. The argument is inspired by \cite[Proposition~5.4]{DMMY25}. Let $G\in\C$ be a graph on $n$ vertices, and set $\gamma:=\beta/(s+1)$, and $q:=\lceil n^\gamma\rceil$. Repeatedly remove a clique of size at least $q+1$ until the remaining induced subgraph $R$ has clique number at most $q$.  Since the removed cliques are pairwise disjoint, there are at most $n/(q+1)=O(n^{1-\gamma})$ of them. On the other hand, applying \((\ast)\) to $R$ gives a balanced separator $S$ with $|S| \le Kq^s n^{1-\beta} = O(n^{1-\gamma})$.
The removed cliques together with the singleton cliques $\{\{v\}:v\in S\}$ therefore form a clique-based separator of $G$ consisting of $O(n^{1-\gamma})$ cliques. 

We conclude with \Cref{algo}. Starting from strongly sublinear tree-independence number, \Cref{equivalence} gives strongly sublinear $\alpha$-separators, which we then upgrade to polynomial $(\dg,\omega)$-boundedness (\Cref{alpha-separators-poly-dgomega}). Thus, there exist constants $A\ge1$ and $s\ge0$ such that $\dg(H)+1\le A\omega(H)^s$ for every $H\in\mathcal C$. Combining this with the elementary inequality $\tw(H)+1\le\chi(H)\tin(H)\le(\dg(H)+1)\tin(H)$ yields polynomial control of the treewidth in terms of the clique number and the tree-independence number. We then apply the clique-peeling argument from the previous paragraph: If $\tin(H)\le a|V(H)|^{1-\delta}$ and $\gamma:=\delta/(s+1)$, then after peeling cliques of size about $n^\gamma$, the remainder $R$ satisfies $\tw(R)=O(n^{1-\gamma})
$. The single-exponential treewidth approximation algorithm of \citet{Kor23} then constructs a small-width decomposition of $R$ in subexponential time. Adding the peeled vertices to every bag yields a tree decomposition of the original graph with strongly sublinear independence number.

\clearpage
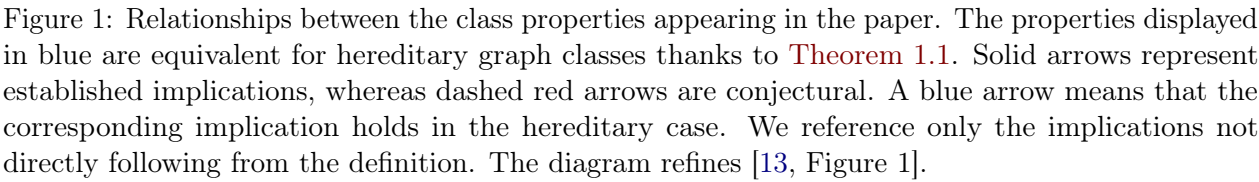
\begin{figure}[p]
\centering
\begin{tikzpicture}[
  x=.82cm,y=.82cm,
  >=Stealth,
  prop/.style={rectangle,draw=sandborder,very thick,fill=sandbox,
               rounded corners=.7pt,align=center,
               inner xsep=2.8pt,inner ysep=2.6pt,
               font=\fontsize{7.2}{7.8}\selectfont},
  wideprop/.style={prop},
  treeprop/.style={prop},
  eqprop/.style={rectangle,draw=blueborder,very thick,fill=bluebox,
                 rounded corners=.7pt,align=center,
                 inner xsep=2.8pt,inner ysep=2.6pt,
                 font=\fontsize{7.2}{7.8}\selectfont},
  known/.style={-{Stealth[length=3.2pt,width=4.2pt]},line width=1.1pt,draw=sandborder},
  open/.style={-{Stealth[length=3.3pt,width=4.4pt]},dashed,
               line width=1.1pt,red!75!black,
               preaction={draw=white,line width=3.5pt,-}},
  blueimp/.style={-{Stealth[length=3.3pt,width=4.4pt]},
                  line width=1.1pt,blue!60!black,
                  preaction={draw=white,line width=3.5pt,-}},
  eqlab/.style={font=\scriptsize,blue!60!black},
  redlab/.style={font=\scriptsize,text=red!75!black,fill=white,inner sep=1.1pt}
]

\node[prop] (bdegree)  at (0.55,0.45)  {\textsf{bounded}\\[-1pt]\textsf{degree}};
\node[prop] (exp)      at (0.55,12.05) {\textsf{bounded}\\[-1pt]\textsf{expansion}};
\node[prop] (degn)     at (0.55,14.25) {\textsf{bounded}\\[-1pt]\textsf{degeneracy}};
\node[prop] (tdegn)    at (0.55,16.45) {\textsf{bounded}\\[-1pt]\textsf{$\theta$-degeneracy}};
\node[prop] (adegn)    at (0.55,18.65) {\textsf{bounded}\\[-1pt]\textsf{$\alpha$-degeneracy}};
\node[prop] (dgnomega) at (0.55,20.85) {\textsf{poly}\\[-1pt]\textsf{$(\mathsf{dgn},\omega)$-bounded}};
\node[prop] (chibound) at (0.55,23.05) {\textsf{poly}\\[-1pt]\textsf{$\chi$-bounded}};

\node[wideprop] (layeredtw) at (3.35,3.40)
  {\textsf{bounded}\\[-1pt]\textsf{layered $\mathsf{tw}$}};
\node[wideprop] (polyftf) at (3.35,6.05)
  {\textsf{fractionally}\\[-1pt]\textsf{$\mathsf{tw}$-fragile}\\[-1pt]\textsf{poly-rate}};
\node[wideprop] (ftf) at (3.35,10.35)
  {\textsf{fractionally}\\[-1pt]\textsf{$\mathsf{tw}$-fragile}};

\node[wideprop] (ssep) at (6.20,8.15)
  {\textsf{strongly sublinear}\\[-1pt]\textsf{separators}};

\node[treeprop] (tw) at (17.60,0.45)
  {\textsf{bounded}\\[-1pt]\textsf{$\mathsf{tw}$}};
\node[treeprop] (tin) at (17.60,2.8)
  {\textsf{bounded}\\[-1pt]\textsf{$\mathsf{tree}\textnormal{-}\alpha$}};
\node[treeprop] (layeredtin) at (17.60,5.15)
  {\textsf{bounded}\\[-1pt]\textsf{layered $\mathsf{tree}\textnormal{-}\alpha$}};
\node[treeprop] (polyftaf) at (17.60,11.35)
  {\textsf{fractionally}\\[-1pt]\textsf{$\mathsf{tree}\textnormal{-}\alpha$-fragile}\\[-1pt]
   \textsf{poly-rate}};
\node[treeprop] (ftaf) at (17.60,16.70)
  {\textsf{fractionally}\\[-1pt]\textsf{$\mathsf{tree}\textnormal{-}\alpha$-fragile}};
\node[treeprop] (cbsep) at (17.60,20.5)
  {\textsf{sublinear-weight}\\[-1pt]\textsf{clique-based separators}};

\def\blueY{14.025}
\def\blueGap{0.42}
\node[eqprop,anchor=west] (cppe) at (3.9,\blueY)
  {\textsf{clique-dep.}\\[-1pt]\textsf{poly}\\[-1pt]\textsf{expansion}};
\node[eqprop,anchor=west] (sepomega) at ($ (cppe.east)+(\blueGap,0) $)
  {\textsf{strongly}\\[-1pt]\textsf{sublinear}\\[-1pt]
   $(\mathsf{sep},\omega)$\textsf{-}\\[-1pt]\textsf{bounded}};
\node[eqprop,anchor=west] (twomega) at ($ (sepomega.east)+(\blueGap,0) $)
  {\textsf{strongly}\\[-1pt]\textsf{sublinear}\\[-1pt]
   $(\mathsf{tw},\omega)$\textsf{-}\\[-1pt]\textsf{bounded}};
\node[eqprop,anchor=west] (cbsize) at ($ (twomega.east)+(\blueGap,0) $)
  {\textsf{strongly}\\[-1pt]\textsf{sublinear}\\[-1pt]
   \textsf{clique-based}\\[-1pt]\textsf{separators}};
\node[eqprop,anchor=west] (stinalpha) at ($ (cbsize.east)+(\blueGap,0) $)
  {\textsf{strongly}\\[-1pt]\textsf{sublinear}\\[-1pt]
   \textsf{$\mathsf{tree}\textnormal{-}\alpha$}};
\node[eqprop,anchor=west] (sasep) at ($ (stinalpha.east)+(\blueGap,0) $)
  {\textsf{strongly}\\[-1pt]\textsf{sublinear}\\[-1pt]
   $\alpha$\textsf{-}\textsf{separators}};

\foreach \a/\b in {cppe/sepomega,sepomega/twomega,twomega/cbsize,
                    cbsize/stinalpha,stinalpha/sasep}{
  \node[eqlab] at ($ (\a.east)!0.5!(\b.west) $) {$\Leftrightarrow$};
}

\draw[known] (bdegree) -- (exp);
\draw[known] (exp) -- (degn);
\draw[known] (degn) -- (tdegn);
\draw[known] (tdegn) -- (adegn);
\draw[known] (adegn) -- node[pos=.5,left,font=\tiny] {\cite{DMMY25}} (dgnomega);
\draw[known] (dgnomega) -- (chibound);
\draw[known,blueborder] (sasep.north west) -- node[pos=.3,left,xshift=-3mm,font=\tiny] {\cref{alpha-separators-poly-dgomega}} (dgnomega.south east);

\draw[known,blueborder] (ssep) .. controls (6,11) .. node[pos=.9,above,font=\tiny] {\cite{DN16}} (exp);
\draw[known] (polyftf.east) -- (polyftaf.south west);
\draw[known] (tw) -- (tin);
\draw[known] (tw.north west) -- (layeredtw.south east);
\draw[known] (tin) -- (layeredtin);
\draw[known] (layeredtw.north east) -- (layeredtin.south west);
\draw[known] (layeredtw) -- (polyftf);
\draw[known] (polyftf) -- (ftf);
\draw[known] (polyftf.north east) -- (ssep.south west);
\draw[known] (layeredtin) -- node[pos=.5,right,font=\tiny] {\cite{GMY23}} (polyftaf);

\draw[known] (ftf.north west) -- node[pos=.2,left,xshift=-2mm,font=\tiny] {\cite{Dvo16}}  (exp.south east);
\draw[known] (ftf.north east)
  .. controls (16.5,12) .. ($(ftaf.south)+(-1.5mm,0)$);
\draw[known] (polyftaf) -- (ftaf);
\draw[known] (ftaf) -- node[midway,right,font=\tiny] {\cref{fractional-treealpha-sublinear-clique-weight}} (cbsep);
\draw[known] (polyftaf.north west) -- node[pos=.15,left,font=\tiny] {\cref{poly-fragility-alpha-separators}} (sasep.south east);
\draw[known] (ssep.north) -- (sepomega.south);
\draw[known] (cbsize.north east) -- (cbsep.south west);
\draw[known] (ftaf.north west) -- node[midway,above,font=\tiny] {\cite{DMMY25}} (dgnomega.east);

\draw[known] (layeredtin.west)
  .. controls (2.5,13) .. node[pos=.08,below left,font=\tiny] {\cite{DMMY25}} (adegn.south east);

\draw[open] (ssep.north west) -- node[midway,right,font=\tiny] {\cite{Dvo16}} (ftf.south east);
\draw[open] (ftaf.north west) -- node[pos=.8,below,font=\tiny] {\cite{DMMY25}} (adegn.east);
\draw[open] (sasep.north east) -- node[midway,left,font=\tiny] {\cref{alphaDvorak}}
  (ftaf.south west);
\end{tikzpicture}
\caption{Relationships between the class properties appearing in the paper. The properties displayed in blue are equivalent for hereditary graph classes thanks to \Cref{equivalence}. Solid arrows represent established implications, whereas dashed red arrows are conjectural. A blue arrow means that the corresponding implication holds in the hereditary case. We reference only the implications not directly following from the definition. The diagram refines \cite[Figure~1]{DMMY25}.}
\label{reldiagram}
\end{figure}
\clearpage

\section{Preliminaries}\label{sec:prelim}

The set of positive integers is denoted by $\mathbb{N}$ and we let $\mathbb{N}_0 := \mathbb{N} \cup \{0\}$. For $k \in \N$, we let $[k]:=\{1,\ldots,k\}$. The base-$2$ logarithm is denoted by $\log$. All graphs in the paper are finite and simple. Let $G$ be a (nonempty) graph and let $v\in V(G)$. 

The neighborhood of $v$ in $G$ is denoted by $N_G(v)$, and the degree of $v$ in $G$ by $d_G(v)$. Let $\delta(G)$ and $\Delta(G)$ denote the
minimum and maximum degree of $G$, respectively. The independence number of $G$ is denoted by $\alpha(G)$, the clique number of $G$ is denoted by $\omega(G)$, the chromatic number of $G$ is denoted by $\chi(G)$, and the clique cover number of $G$ is denoted by $\theta(G)$. For $X\subseteq V(G)$, the
notation $G[X]$ denotes the subgraph of $G$ induced by $X$. With $\overline G$ we denote the complement of $G$.

The \emph{average degree} of $G$, denoted $\ad(G)$, is the quantity $2|E(G)|/|V(G)|$. The \emph{degeneracy} of $G$,  denoted $\dg(G)$, is the maximum, over all nonnull induced subgraphs $H$ of $G$, of the minimum degree of $H$.
The \emph{$\alpha$-degeneracy} of $G$ is the maximum, over all nonnull induced subgraphs $H$ of $G$, of the minimum, over all vertices $v\in V(H)$, of $\alpha(H[N_H(v)])$.
Equivalently, the $\alpha$-degeneracy of $G$ is the smallest integer $k$ such that every nonnull induced subgraph of $G$ has a vertex that is not the center of an induced star with more than $k$ leaves. This parameter was dubbed \emph{inductive independence number} by \citet{YB12}. The \textit{$\theta$-degeneracy} of $G$ is defined similarly to the $\alpha$-degeneracy of $G$, except that the independence number is replaced by the clique cover number. 

A graph class $\C$ is \emph{hereditary}
if it is closed under taking induced subgraphs, and \emph{monotone} if it is closed under taking subgraphs. A graph class $\mathcal{G}$ is said to be \emph{$\chi$-bounded} if it admits a \emph{$\chi$-binding function}, that is, a function $f$ such that for every graph $G$ in the class and every induced subgraph $H$ of $G$, it holds that $\chi(H)\le f(\omega(H))$ (see~\cite{Gyarfas87,SS20}).
Furthermore, if there exists a polynomial function with this property, then the class is said to be \emph{polynomially $\chi$-bounded}.    

A restricted version of $\chi$-boundedness is the following.
We say that a graph class $\mathcal{G}$ is \emph{$(\dg,\omega)$-bounded} 
(resp., \emph{polynomially $(\dg,\omega)$-bounded}) if there exists a function $f$ (resp., a polynomial function) such that for every graph $G$ in the class and every induced subgraph $H$ of $G$, it holds that $\dg(H)\le f(\omega(H))$.
While not every $\chi$-bounded graph class is polynomially $\chi$-bounded \cite{BDW24}, the following consequence of \cite{GH25} and \cite[Corollary~3.5]{DMcC25} holds.  

\begin{theorem}\label{poly-dgnomega}
Every $(\dg,\omega)$-bounded graph class is polynomially $(\dg,\omega)$-bounded.     
\end{theorem}

Let $G$ and $H$ be graphs and let $r \geq 0$ be an integer. $H$ is a \emph{minor} of $G$ if a graph isomorphic to $H$ can be obtained from $G$ by vertex deletion, edge deletion, and edge contraction. A \emph{model} of $H$ in $G$ is a family $\{B_v\}_{v\in V(H)}$ of pairwise disjoint vertex subsets of $V(G)$, called \emph{branch sets}, such that each $G[B_v]$ is connected and, for every edge $uv\in E(H)$, there is an edge of $G$ with one endpoint in $B_u$ and the other in $B_v$. It is well known that $H$ is a minor of $G$ if and only if $G$ contains a model of $H$. If there exists a model $\{B_v\}_{v\in V(H)}$ of $H$ in $G$ such that each $G[B_v]$ has radius at most $r$, then such a model is an \emph{$r$-shallow minor model} of $H$ in $G$ and $H$ is an
\emph{$r$-shallow minor} of $G$. Recall here that the \emph{radius} of $G$ is the quantity $\min_{x\in V(G)}\max_{y\in V(G)}d_G(x,y)$. We let $\omegar(G):=\max\{t:K_t\text{ is an $r$-shallow minor of }G\}$.

For a graph $G$ and an integer $r\ge0$, let $\nabla_r(G)$ denote the maximum of $|E(H)|/|V(H)|$ over all nonempty $r$-shallow minors $H$ of $G$, with $\nabla_r(G):=0$ when $G$ is empty. A class $\mathcal C$ has \emph{bounded expansion} if there exists a function
$f\colon\mathbb N_0\to\mathbb R_{\ge0}$ such that
$\nabla_r(G)\le f(r)$ for every $G\in\mathcal C$ and
$r\in\mathbb N_0$; it has \emph{polynomial expansion} if such a function $f$ can be chosen to be a polynomial.

A \textit{tree decomposition} of a graph $G$ is a pair $\mathcal{T} = (T, \{X_t\}_{t\in V(T)})$, where $T$ is a tree whose every node $t$ is assigned a vertex subset $X_t \subseteq V(G)$, called \textit{bag}, such that the following conditions are satisfied: 
\begin{description}
\item[(T1)] Every vertex of $G$ belongs to at least one bag; 
\item[(T2)] For every $uv \in E(G)$, there exists a bag containing both $u$ and $v$; 
\item[(T3)] For every $u \in V(G)$, the subgraph $T_u$ of $T$ induced by $\{t \in V(T)\colon u \in X_t\}$ is connected. 
\end{description}
The \textit{width} of $\mathcal{T} = (T, \{X_t\}_{t\in V(T)})$ is the maximum value of $|X_t| - 1$ over all $t \in V(T)$. The \textit{treewidth} of a graph $G$, denoted $\tw(G)$, is the minimum width of a tree decomposition of $G$. The \textit{independence number} of $\mathcal{T}$, denoted $\alpha(\mathcal{T})$, is the quantity $\max_{t\in V(T)} \alpha(G[X_t])$. The \textit{tree-independence number} of a graph $G$, denoted $\tin(G)$, is the minimum independence number of a tree decomposition of $G$ (see \cite{DMS24}). For a tree decomposition $\mathcal{T} = (T, \{X_t\}_{t\in V(T)})$, we let $|\mathcal{T}| = |V(T)|$. We will use the following standard property of tree decompositions (see \cite[Lemma~7.19]{CFK}). 

\begin{lemma}[Folklore]\label{balanced-bag}
Let $G$ be a nonempty $n$-vertex graph and let $\mathcal T=(T,\{X_t\}_{t\in V(T)})$ be a tree decomposition of $G$. There exists $t_0\in V(T)$ such that every connected component of $G-X_{t_0}$ has at most $n/2$ vertices.
\end{lemma}

We also recall two properties of treewidth. First, treewidth is monotone under minors. Second, the following well-known result holds, where the lower bound is folklore and the upper bound is due to \citet{DN19}. Here the \emph{separation number} $\sepnum(G)$ of a graph $G$ is the smallest integer $s$ such that every subgraph of $G$ has a balanced separator of size at most $s$.

\begin{lemma}\label{DNsep} For every graph $G$, it holds that 
    $\mathsf{sep}(G)-1\leq \tw(G)\le 15\sepnum(G)$.
\end{lemma}

The following elementary observation will be repeatedly used.

\begin{observation}[Folklore]\label{subgraph-separator}
Let $\C$ be a hereditary class and let $G\in\C$. If every $n$-vertex graph in
$\C$ has a balanced separator of size at most $f(n)$, where $f$ is
nondecreasing, then every subgraph $H$ of $G$ has a balanced separator of
size at most $f(|V(H)|)$.
\end{observation}

\begin{proof}
Let $J:=G[V(H)]$. Since $\C$ is hereditary, $J\in\C$, and hence $J$ has a
balanced separator $S$ of size at most $f(|V(H)|)$. The graph $H-S$ is a
spanning subgraph of $J-S$, so every component of $H-S$ is contained in a
component of $J-S$. Thus $S$ is also a balanced separator of $H$.
\end{proof}

It is well known that hereditarily admitting strongly sublinear separators implies bounded average degree and hence bounded degeneracy. We record this as follows and provide the short proof for completeness.

\begin{lemma}[\citet{DN16}]\label{degeneracy-from-sep}
Let $\C$ be a hereditary class. Suppose there exist constants $c\ge1$ and $\varepsilon\in(0,1]$
such that every $n$-vertex graph in $\C$ has a balanced separator of size at
most $cn^{1-\varepsilon}$. Then there is an integer $d_{c,\varepsilon}$
such that every graph in $\C$ is $d_{c,\varepsilon}$-degenerate.
\end{lemma}

\begin{proof}
Fix $G\in\C$ and a subgraph $H$ of $G$. By
\cref{subgraph-separator}, every subgraph $J$ of $H$ has a balanced
separator of size at most $c|V(J)|^{1-\varepsilon}$. By \cite[Lemma~12]{DN16}, there exists a constant $a_{c,\varepsilon}$, independent of $H$, such that $|E(H)|\le a_{c,\varepsilon}|V(H)|$.
Thus every nonempty subgraph $H$ of $G$ has average degree at most $2a_{c,\varepsilon}$ and hence contains a vertex of degree at most $2a_{c,\varepsilon}$. 
\end{proof}

A \emph{layering} of a graph $G$ is an ordered partition
$(V_i)_{i\in\mathbb Z}$ of $V(G)$ such that, for every edge $uv$ with
$u\in V_i$ and $v\in V_j$, we have $|i-j|\le1$.
Let $\mathcal{T} = (T,\{X_t\}_{t \in V(T)})$ be a tree decomposition of a graph $G$.
The \emph{layered width} of $\mathcal{T}$ is the minimum integer $\ell$ such that there exists a layering $(V_i)_{i\in\mathbb Z}$ of $G$ such that for each bag $X_t$ and each layer $V_i$, it holds that $|X_t\cap V_i|\le \ell$.
Similarly, the \emph{layered independence number} of $\mathcal{T}$ is the minimum integer $\ell$ such that there exists a layering $(V_i)_{i\in\mathbb Z}$ of $G$ such that for each $t\in V(T)$ and each $i\in\mathbb{Z}$, it holds that $\alpha(G[X_t\cap V_i])\le \ell$.
The \emph{layered treewidth} (resp., \emph{layered tree-independence number}) of a graph $G$, denoted $\mathsf{ltw}(G)$ (resp., $\ltreealpha(G)$), is defined as the minimum layered width (resp., layered independence number) of a tree decomposition of $G$.
The notions of layered treewidth and layered tree-independence number were introduced by \citet{DMW17} and \citet{GMY24}, respectively.  

\section{The induced sparsification lemma}\label{sec:sparsification}

In this section, we prove the following key induced sparsification result (\Cref{induced-sparsification}): If a $d$-degenerate graph contains a sufficiently large $r$-shallow clique minor, then one can extract from it an induced $K_7$-free subgraph of controlled order and large treewidth. 

In order to do so, we need two elementary lemmas. The first deals with inclusion-minimal connected induced subgraphs containing a prescribed terminal set. \citet[Lemma~4]{CGKMUW19} showed that if there are $t$ terminals then every vertex of such a connector has degree at most $t$. We show the following sharper clique-number estimate. 

\begin{lemma}\label{terminal-connector}
Let $G$ be a connected graph and let $\varnothing\neq Z\subseteq V(G)$. Suppose that no proper induced subgraph of $G$ containing $Z$ is connected. Then $\omega(G)\le |Z|$.
\end{lemma}

\begin{proof}
Let $K$ be a clique of $G$. We show that $|K|\le|Z|$. Since $Z\neq\varnothing$, the assertion is immediate when $|K|\le1$. Hence, assume that $|K|\ge2$. For every $v\in K\setminus Z$, minimality implies that $G-v$ is disconnected. Since $K\setminus\{v\}$ is a nonempty clique, it belongs to a unique component $M_v$ of $G-v$. Let $C_v$ be a distinct component of $G-v$. Observe that $C_v$ contains a vertex of $Z$. Indeed, if this is not the case, then $G-V(C_v)$ is a proper induced subgraph of $G$ containing $Z$ and which is connected. Choose $z_v\in Z\cap V(C_v)$ and note that $z_v\notin K$.

We now claim that the vertices $z_v$, for $v\in K\setminus Z$, are pairwise
distinct. Suppose, to the contrary, that $v,w\in K\setminus Z$ are distinct but $z_v=z_w=:z$. Since $z\in C_v$, every $z,w$-path in $G$ contains $v$, as $w\in K\setminus\{v\}\subseteq M_v$. On the other hand, since $z\in C_w$, there is a $z,w$-path in $G$ whose internal vertices belong to $C_w$. This path avoids $v$, as $v\in K\setminus\{w\}$ lies in the component $M_w$ of $G-w$ distinct from $C_w$, a contradiction.

Thus the map $v\mapsto z_v$ injects $K\setminus Z$ into $Z\setminus K$.
Consequently, $|K\setminus Z|\le |Z\setminus K|$, and so $|K| =|K\cap Z|+|K\setminus Z| \le |K\cap Z|+|Z\setminus K| =|Z|$.
\end{proof}

The second lemma is an elementary clique count for $d$-degenerate graphs.

\begin{lemma}\label{count-k7}
If $G$ is $d$-degenerate, then the number of copies (as a subgraph) of $K_7$ in $G$ is at most $|V(G)|\binom d6$.
\end{lemma}

\begin{proof}
Take a degeneracy ordering $v_1,\ldots,v_n$ of $G$, so that every vertex has at most $d$ later neighbors. Charge each copy of $K_7$ to its earliest vertex in the ordering. Once the charged vertex is fixed, the other six vertices must be chosen among its later neighbors, giving at most $\binom d6$ possibilities.
\end{proof}

We also require the following known results on expander graphs. For $\alpha>0$, a graph $G$ is an \emph{$\alpha$-expander} if every set $A\subseteq V(G)$ with $|A|\le |V(G)|/2$ has at least $\alpha|A|$ neighbors
in $V(G)\setminus A$. 

\begin{lemma}[\citet{KW14}; see also \cite{DN16}]\label{expanders}
There exists an integer $n_0$ such that for every even integer $n\ge n_0$
there exists a $3$-regular $1/7$-expander on $n$ vertices.
\end{lemma}

\begin{lemma}[\citet{DN16}]\label{expander-tw} Let $\alpha > 0$. If $G$ is an $\alpha$-expander, then \[\tw(G)\ge \frac{\alpha}{3(1+\alpha)}|V(G)|-1.\]
\end{lemma}

The proof of \Cref{induced-sparsification} exploits the redundancy in a large $r$-shallow clique model as follows. We first partition $nL$ branch sets into $n$ blocks $\mathcal A_x$ of size $L$ (to be fixed later), indexed by the $n$ vertices of a fixed cubic expander $F$. For every edge $xy\in E(F)$ and every pair $B\in\mathcal A_x$, $B'\in\mathcal A_y$, we fix an edge joining $B$ and $B'$. Using the radius-$r$ assumption, we retain inside each branch set only $O(L(r+1))$ vertices needed to connect all these prescribed attachment points. The induced subgraph $J_0$ on all these retained vertices has then $O(nL^2(r+1))$ vertices.

We then choose independently and uniformly at random one branch set from each block. The prewiring above guarantees that the selected branch sets still contain $F$ as a minor. Now, since $F$ is cubic, each selected branch set contains at most three attachment points. Hence, the subgraph induced by an inclusion-minimal connector for these attachment points has clique number at most three by \Cref{terminal-connector}. This implies that every $K_7$ in the subgraph $J$ induced by the selected connectors must meet at least three selected branch sets. A fixed $K_7$ of $J_0$ therefore survives in $J$ with probability at most $1/L^{3}$. Since $J_0$ is $d$-degenerate, \Cref{count-k7} then implies that the expected number of surviving $K_7$'s is $O_d(n(r+1)/L)$. Choosing $L=\Theta_d(r+1)$, we can then delete one vertex from each surviving $K_7$ while losing only a small linear amount of treewidth.

Note that the probabilistic mechanism we use is reminiscent of Bonnet's induced-sparsification argument~\cite[Lemma~16]{Bonnet2025}. However, the implementations are different.

\begin{lemma}\label{induced-sparsification}
There is an absolute integer $M\ge1$ with the following property. Let
$d,r\in\mathbb N_{0}$ and $t\in\N$, and let $G$ be a $d$-degenerate graph containing $K_t$ as an $r$-shallow minor. Let $b_d:=\max\left\{1,\binom d6\right\}$ and $L:=\left\lceil 576\,b_d(r+1)\right\rceil$. If $t\ge2ML$,
then $G$ contains an induced subgraph $H$ such that
\begin{enumerate}[label=(\roman*)]
  \item $\omega(H)\le6$
  \item $|V(H)|\le \frac{3t(r+1)}{L}$
  \item $\tw(H)\ge \frac{t}{128L}$.
\end{enumerate}
\end{lemma}

\begin{proof}
Let $n_0$ be as in \Cref{expanders}, and set $M:=\max\{n_0,48\}$.
Let $\{B_i\}_{i\in[t]}$ be an $r$-shallow minor model of $K_t$ in $G$.
Thus the branch sets $B_i \subseteq V(G)$ are pairwise disjoint, every $G[B_i]$ is connected and has radius at most $r$, and every two distinct branch sets are joined by an edge of $G$.

Set $T:=\left\lfloor t/L \right\rfloor$. Since $t\ge2ML$, we have $T\ge2M$. Let $n$ be the largest even integer at most $T$. Thus $n\ge 2M$ and, by \Cref{expanders}, there exists a $3$-regular $1/7$-expander $F$ on $n$ vertices. Applying \Cref{expander-tw} with $\alpha = 1/7$ gives $\tw(F)\ge \frac n{24}-1$. Observe now that, by definition of $n$ and $T$, we have $n\le T\le t/L$ and $T\le n+1$. Moreover, $t<L(T+1)\le L(n+2)\le2Ln$, where the last inequality uses $n\ge M\ge48$. Combining these, we get 
\begin{equation}\label{eq:m-window}
  \frac{t}{2L}<n\le\frac{t}{L}.
\end{equation}

Choose $nL$ of the branch sets and partition them into $n$ groups $\mathcal A_x$, for $x\in V(F)$, each consisting of exactly $L$ branch sets. For every edge $xy\in E(F)$ and every pair $B\in\mathcal A_x$ and $B'\in\mathcal A_y$,
choose one edge of $G$ with one endpoint in $B$ and the other in $B'$.
Note that such an edge exists because the original branch sets form a clique minor model. Let $E_0\subseteq E(G)$ be the set of the chosen edges. 

Fix now a branch set $B\in\mathcal A_x$. Since $F$ is $3$-regular, $B$ is
incident with at most $3L$ edges from $E_0$ and let $Z_B$ be the set of their
endpoints in $B$. Hence, $|Z_B|\le3L$. Choose a center $c_B\in B$ such
that every vertex of $B$ is at distance at most $r$ from $c_B$ in $G[B]$.
For every $z\in Z_B$, choose a $c_B,z$-path in $G[B]$ of length at most
$r$, and let $C_B$ be the union of the vertex sets of these paths (if
$Z_B=\varnothing$, put $C_B:=\{c_B\}$). Hence, $|C_B|\le1+3Lr$. Note also that every edge in $E_0$ has both endpoints in the corresponding sets $C_B$.

Define now the induced subgraph
\[
  J_0:=G\left[\bigcup_{x\in V(F)}\ \bigcup_{B\in\mathcal A_x} C_B\right].
\]
We have that $|V(J_0)|\leq nL(1+3Lr)\le 3nL^2(r+1)$.
Since $J_0$ is an induced subgraph of the $d$-degenerate graph $G$, it is
$d$-degenerate. Hence by \Cref{count-k7}, the number $N_7(J_0)$ of
copies of $K_7$ in $J_0$ satisfies
\begin{equation}\label{eq:N7J0}
  N_7(J_0) \le |V(J_0)|\binom d6 \le 3nL^2(r+1)b_d.
\end{equation}

We now perform a random thinning. Independently for every
$x\in V(F)$, choose one branch set $B_x$ uniformly from $\mathcal A_x$.
For every neighbor $y$ of $x$ in $F$, let $e_{xy} \in E_0$ denote the edge that was chosen for the pair $B_x,B_y$, and let $z_{xy}$ be its
endpoint in $B_x$. Finally, set $Z_x:=\{z_{xy}:y\in N_F(x)\}$.
Since $F$ is $3$-regular, $1 \leq |Z_x|\le3$.
By construction of $C_{B_x}$, the set $Z_x$ together with the center
$c_{B_x}$ is contained in a connected subgraph of $G[B_x]$ on at most
$1+3r$ vertices, namely the union $U_x$ of the vertex sets of the at most three selected $c_{B_x},Z_x$-paths. Among all subsets $W\subseteq U_x$ containing $Z_x$ and such that $G[W]$ is connected, choose an inclusion-minimal one $W_x$. If there is more than one such set, choose one according to a fixed deterministic tie-breaking rule. Then $|W_x|\le1+3r\leq3(r+1)$
and, by \Cref{terminal-connector}, $\omega(G[W_x])\le |Z_x| \le3$.

Let now
\[
  J:=G\left[\bigcup_{x\in V(F)}W_x\right].
\]
The sets $W_x$ are pairwise disjoint, as they lie in distinct original
branch sets. For every edge $xy\in E(F)$, the selected edge
$e_{xy}\in E_0$ has one endpoint in $W_x$ and the other in $W_y$. Therefore,
$\{W_x\}_{x\in V(F)}$ is a minor model of $F$ in $J$. Consequently,
\begin{equation}\label{eq:Jtw}
  \tw(J)\ge\tw(F)\ge\frac n{24}-1.
\end{equation}
Also, 
\begin{equation}\label{eq:Jsize}
  |V(J)|\le3n(r+1).
\end{equation}

Let $Z$ be the random variable counting the number of copies of $K_7$ in
$J$. We claim that every fixed copy $Q\cong K_7$ in $J_0$ survives in $J$
with probability at most $1/L^{3}$. Indeed, if $Q$ is contained in $J$ then, since $\omega(G[W_x])\le3$, the clique $Q$ contains at most three vertices from each selected branch set $B_x$. Hence, $Q$ meets at least three selected branch
sets. If two of those branch sets belong to the same group $\mathcal A_x$,
then the survival probability is $0$, as only one branch set is selected
from each group. Otherwise, $Q$ meets at least three branch sets belonging to
distinct groups, and all of them must be selected. By independence, this has
probability at most $1/L^{3}$. Therefore, using \eqref{eq:N7J0}, we obtain
\begin{equation*}
  \mathbb E[Z]
  \le \frac{N_7(J_0)}{L^3}
  \le \frac{3n(r+1)b_d}{L}
  \le \frac n{192}.
\end{equation*}

Hence, there is a deterministic choice of the branch sets for which $Z\le n/192$.
Fix such a choice. Select one vertex from each copy of $K_7$ in $J$, and let
$X$ be the set of selected vertices. Clearly, $|X|\le n/192$.
We finally set $H:=J-X$. Since $J$ is induced in $G$, the graph $H$ is also an induced subgraph of $G$. The set $X$ meets every copy of $K_7$ in $J$, hence $H$ is $K_7$-free and so $\omega(H)\le6$.
Moreover, by \eqref{eq:Jsize} and the upper bound in \eqref{eq:m-window},
\[
  |V(H)|\le |V(J)|\le3n(r+1) \le \frac{3t(r+1)}{L}.
\]
Finally, using the well-known\footnote{From a tree decomposition of $J-X$ we obtain a tree decomposition of $J$ by adding every vertex of $X$ to every bag.} inequality $\tw(J)\le\tw(J-X)+|X|$ and recalling \eqref{eq:Jtw} and that $|X| \leq n/192$, we obtain 
\[
  \tw(H)
  \ge \tw(J)-|X|
  \ge \left(\frac n{24}-1\right)-\frac n{192}
  >\frac n{64}
  >\frac{t}{128L},
\]
where the second to last inequality follows from $n > 48$ and the last inequality follows from the lower bound in \eqref{eq:m-window}. This completes the proof.
\end{proof}

\section{Uniformization of clique-dependent separators}\label{sec:uniformization}

In this section, we combine the induced sparsification lemma (\Cref{induced-sparsification}) with the Plotkin--Rao--Smith theorem to obtain strongly sublinear separators with a power saving uniform in the clique number, while keeping polynomial control of the multiplicative dependence on the clique number. This is the content of \Cref{poly-coefficient}, which gives $\textup{(ii)}\Rightarrow\textup{(iv)}$ in \Cref{equivalence}.

The induced sparsification lemma allows us to transfer the separator bound
from the clique-$6$ slice to shallow clique minors in every bounded-clique
slice. Indeed, a large shallow clique minor would produce an induced
$K_7$-free subgraph of large treewidth. But this induced subgraph belongs to $\C_{\le6}$, where strongly sublinear separators give an opposite
treewidth bound. Comparing the two bounds yields a polynomial bound on the
order of shallow clique minors. We formalize this in the following. Recall that $\omegar(G):=\max\{t:K_t\text{ is an $r$-shallow minor of }G\}$.

\begin{proposition}\label{shallow-clique} Let $\C$ be a hereditary class. Suppose that there exist $c\ge 1$ and $\eta\in(0,1]$ such that every $n$-vertex graph in $\C_{\le6}$ has a balanced separator of size at most $c n^{1-\eta}$. Suppose further that, for every $q\in \mathbb{N}$, there exists $d_q \in \mathbb{N}$ such that every graph in $\C_{\le q}$ is $d_q$-degenerate. Then there is a constant $C_{c,\eta}\ge1$ such that, for every $r\ge0$ and every $G\in\C_{\le q}$,
\begin{equation*}
  \omegar(G)\le C_{c,\eta}(1+d_q^6)(r+1)^{\frac1\eta}.
\end{equation*}
\end{proposition}

\begin{proof}
Fix $q\in\mathbb{N}$, and let $d:=d_q$. Let $G\in\C_{\le q}$ and suppose that $K_t$ is an $r$-shallow minor of $G$. We first record a treewidth upper bound for graphs in $\C_{\le6}$. Let $H\in\C_{\le6}$ be a graph on $n$ vertices. By \Cref{subgraph-separator}, every subgraph $J$ of $H$ has a balanced separator of size at most $c|V(J)|^{1-\eta}\le cn^{1-\eta}$. Hence, by \cite[Theorem~1]{DN19}, 
\begin{equation}\label{eq:C6tw}
  \tw(H)\le15cn^{1-\eta}.
\end{equation}
Let $M\ge1$, $b_d:=\max\left\{1,\binom{d}{6}\right\}$, and $L:=\left\lceil576b_d(r+1)\right\rceil$ be the constants from \Cref{induced-sparsification}. 

If $t<2ML$, then 
\[
  t< 2ML \le 2M(576b_d+1)(r+1).
\]

Suppose instead that $t\ge2ML$. By \Cref{induced-sparsification}, there is
an induced subgraph $H$ of $G$ such that $\omega(H)\le6$, $|V(H)|\le 3t(r+1)/L$, and $\tw(H)\ge t/(128L)$.
Since $\C$ is hereditary, $H\in\C_{\le6}$. Applying \eqref{eq:C6tw} gives
\[
  \frac{t}{128L}
  \le15c\left(\frac{3t(r+1)}{L}\right)^{1-\eta}
\]
which, after rearranging, yields $t^{\eta} \le 1920c\,3^{1-\eta}L^{\eta}(r+1)^{1-\eta}$.
Thus, setting $A:=\bigl(1920c\,3^{1-\eta}\bigr)^{1/\eta}$ and since $L\le (576b_d+1)(r+1)$, we obtain 
\[
  t\le A L(r+1)^{(1-\eta)/\eta} \le A(576b_d+1)(r+1)^{\frac{1}{\eta}}.
\]

Therefore, combining the two cases, and since $1/\eta\ge1$ and
$b_d\le1+d^6$, we obtain that there exists $C_{c,\eta}\ge1$ such that
\[ \omegar(G)\le C_{c,\eta}(1+d^6)(r+1)^{1/\eta},\] for every $r\ge0$ and $G\in\C_{\le q}$.
\end{proof}

We next convert the shallow-minor bound resulting from \Cref{shallow-clique} back into a separator bound. This is a direct application of the Plotkin--Rao--Smith theorem, which we record as stated in \cite[Theorem~1]{Dvorak2021}.

\begin{theorem}[\citet{PRS94}]\label{PRS}
Let $G$ be an $n$-vertex graph with $n\ge2$, and let $\ell,h\in\mathbb{N}$ be
integers. Then either
\begin{enumerate}[label=(\roman*)]
  \item $G$ has a balanced separator of size at most $\frac n\ell+2h^2\ell\log n$, or
  \item $G$ contains $K_h$ as a $(2\ell\log n)$-shallow minor.
\end{enumerate}
\end{theorem}

\begin{lemma}\label{shallow-to-separator}
Let $0<\eta\le 1$ and $B\ge1$. Let $G$ be an $n$-vertex graph with $n\ge2$ and suppose that $\omegar(G)\le B(r+1)^{1/\eta}$
for every $r\ge0$. Then there is a constant $C_{\eta}\ge1$ such that
\[
  \mathsf{sep}(G)\le C_{\eta} B^2 n^{1-\frac{\eta}{2\eta+2}}\log n.
\]
\end{lemma}

\begin{proof}
Set $a:=\eta/(2\eta+2)$ and $\ell:=\left\lfloor n^a/(2\log n)\right\rfloor$.
For all sufficiently large $n$, we have $n^a/(4\log n) \le \ell \le n^a/(2\log n)$.
Set $R:=\left\lceil2\ell\log n\right\rceil$. Then $R\le n^a+1$. Now let
$h:=\left\lfloor B(R+1)^{1/\eta}\right\rfloor+1$.
Since $h>B(R+1)^{1/\eta}$, the assumption implies that $G$ has no
$R$-shallow $K_h$ minor. As $2\ell\log n\le R$, the graph $G$ has no
$(2\ell\log n)$-shallow $K_h$ minor either. Hence, by \Cref{PRS}, $G$ has a balanced separator of size at most
\[
  \frac n\ell+2h^2\ell\log n.
\]
The first term satisfies 
\[
  \frac{n}{\ell}
  \le 4n^{1-a}\log n.
\]
Moreover, since $R\le n^a+1$, we have $h=O_\eta(Bn^{a/\eta})$. Using $\ell\le n^a/(2\log n)$, we then obtain
\[
  2h^2\ell\log n
  =O_\eta\!\left(B^2n^{2a/\eta+a}\right)
  =O_\eta\!\left(B^2n^{1-a}\right).
\]
Consequently, there exist constants $N_\eta\ge2$ and $C_\eta>0$, depending only on $\eta$, such that
\[
  \mathsf{sep}(G)\le C_\eta B^2 n^{1-a}\log n
\]
whenever $n\ge N_\eta$. For $2\le n<N_\eta$, the trivial separator $V(G)$ gives $\mathsf{sep}(G)\le n$. Since $B\ge1$, increasing $C_\eta$ if necessary makes the same bound valid for every $n\ge2$.
\end{proof}

Combining \Cref{shallow-clique} with \Cref{shallow-to-separator} gives the following bootstrap statement, which might be of independent interest. The strongly sublinear exponent coming from the clique-$6$ slice propagates, up to a fixed loss, to every bounded-clique slice.

\begin{proposition}\label{bootstrap}
Let $\C$ be a hereditary class. Suppose that there exist $c\ge 1$ and
$\eta\in(0,1]$ such that every $n$-vertex graph in $\C_{\le6}$ has a
balanced separator of size at most $c n^{1-\eta}$. Suppose further that, for every $q\in \mathbb{N}$, there exists $d_q \in \mathbb{N}$ such that every graph in $\C_{\le q}$ is $d_q$-degenerate. Then, for every $q\in\mathbb{N}$, every $n$-vertex graph in $\C_{\le q}$ with $n \geq 2$ has a balanced separator of size \[O_q\!\left(n^{1-\frac{\eta}{2(1+\eta)}}\log n\right).\]
\end{proposition}

\begin{proof}
Fix $q\in\mathbb{N}$. By \Cref{shallow-clique}, there exists $B_q\ge1$
such that $\omegar(G)\le B_q(r+1)^{1/\eta}$ for every $r\ge0$ and every $G\in\C_{\le q}$. Finally, apply \Cref{shallow-to-separator} with $B=B_q$. The claimed bound follows.
\end{proof}

Under the assumptions of \Cref{bootstrap}, for every fixed $\delta< \eta/(2\eta+2)$, there is a constant $C_{q,\delta}$ such that every $n$-vertex graph in $\C_{\le q}$ has a balanced separator of size at most $C_{q,\delta}n^{1-\delta}$. However, for the application to clique-based separators, a uniform power saving in $|V(G)|$ is not enough: we also need quantitative control of the multiplicative dependence on $\omega(G)$. Since the bounded-clique slices have bounded degeneracy, polynomial $(\dg,\omega)$-boundedness gives $\dg(G)\le p(\omega(G))$ for a fixed polynomial $p$. Retaining this dependence in \Cref{shallow-clique} and in the Plotkin--Rao--Smith step makes the separator coefficient polynomial in $\omega(G)$ as well, as we show next.

\begin{theorem}\label{poly-coefficient}
Let $\C$ be a hereditary class. Suppose that, for every $q\in\mathbb{N}$, there exist $c_q\ge1$ and $\eta_q\in(0,1]$ such that every $n$-vertex graph in $\C_{\le q}$ has a balanced separator of size at most $c_q n^{1-\eta_q}$. Then there exist constants $K\ge1$, $\beta>0$, and an integer $s\ge0$ such that, for every $G\in\C$, 
\begin{equation*}
\mathsf{sep}(G) \leq K\omega(G)^s |V(G)|^{1-\beta}.
\end{equation*}
\end{theorem}

\begin{proof} By \Cref{degeneracy-from-sep}, for every $q\ge1$ there exists an integer $e_q$ such that every graph in $\C_{\le q}$ is $e_q$-degenerate. It follows that $\C$ is $(\dg,\omega)$-bounded and so, by \Cref{poly-dgnomega}, polynomially $(\dg,\omega)$-bounded. Hence, there exists a polynomial $g$ such that $\dg(G)\le g(\omega(G))$ for every $G\in\C$. We may assume that $g$ is nondecreasing on $[0,\infty)$. Apply now \Cref{shallow-clique} with $d_q:=\lceil g(q)\rceil$. We obtain a constant $C_{c_6,\eta_6}\ge1$ such that, with the fixed nondecreasing
polynomial $f(x):=C_{c_6,\eta_6}(1+(g(x)+1)^6)$,
we have
\begin{equation*}
  \omegar(G)\le f(q)(r+1)^{\frac{1}{\eta_6}},
\end{equation*}
for every $q\ge1$, every $G\in\C_{\le q}$, and every $r\ge0$.

Let $G\in\C$ be a graph on $n\ge2$ vertices and set $q:=\omega(G)$. Applying
\Cref{shallow-to-separator} with $B=f(q)$ gives
\[
  \mathsf{sep}(G) \le C_{\eta_6} f(q)^2 n^{1-\frac{\eta_6}{2\eta_6+2}}\log n,
\]
for some constant $C_{\eta_6}\geq1$ depending only on $\eta_6$. But then, for every fixed $0 < \beta < \eta_6/(2\eta_6+2)$ and sufficiently large $n$, the right-hand side is at most a fixed polynomial in $q$ times $n^{1-\beta}$. Since every fixed polynomial in $q$ is at most $Kq^s$ for suitable $K\ge1$ and $s\in\mathbb{N}$, increasing the constant if necessary gives the desired bound.
\end{proof}

We conclude this section with some consequences of \Cref{poly-coefficient} (in fact, these can be directly deduced from \Cref{bootstrap} as well). By \Cref{degeneracy-from-sep}, hereditarily admitting strongly sublinear separators implies bounded degeneracy. Therefore, \Cref{poly-coefficient} immediately gives the following uniformization result.

\begin{corollary}\label{uniformization} Let $\C$ be a hereditary class. Suppose that, for every $q\in\mathbb{N}$, there exist $c_q\ge1$ and $\varepsilon_q\in(0,1]$ such that every $n$-vertex graph in $\C_{\le q}$ has a balanced separator of size at most $c_q n^{1-\varepsilon_q}$.
Then there exists a constant $\varepsilon > 0$ such that, for every $q\in\mathbb{N}$, there is a constant $C_q$ such that every $n$-vertex graph in $\C_{\le q}$ has a balanced separator of size at most $C_q n^{1-\varepsilon}$.
\end{corollary}

The same conclusion holds in treewidth language thanks to \Cref{DNsep}.

\begin{corollary}\label{treewidth}
Let $\C$ be a hereditary class. Suppose that, for every $q\in\mathbb{N}$, there exist $a_q\ge1$ and $\varepsilon_q\in(0,1]$ such that $\tw(G)\le a_q |V(G)|^{1-\varepsilon_q}$ for every $G\in\C_{\le q}$. Then there exists a constant $\varepsilon>0$ such that, for every $q\in\mathbb{N}$, there is a constant $b_q$ with $\tw(G)\le b_q |V(G)|^{1-\varepsilon}$ for every $G\in\C_{\le q}$.
\end{corollary}

\begin{proof}
Let $G\in\C_{\le q}$ be a nonempty graph. \Cref{DNsep} gives $\mathsf{sep}(G)\le\tw(G)+1 \le (a_q+1)|V(G)|^{1-\varepsilon_q}$. Thus, \Cref{uniformization} yields a uniform exponent $\varepsilon>0$ for balanced separators. Fix now $q$ and $G\in\C_{\le q}$ with $n$ vertices. By \Cref{subgraph-separator}, every subgraph $H$ of $G$ has a balanced separator of size at most $C_qn^{1-\varepsilon}$. Hence, $\sepnum(G)\le C_qn^{1-\varepsilon}$, and another application of \Cref{DNsep} gives $\tw(G)\le15C_qn^{1-\varepsilon}$.
\end{proof}

\section{The proof of \Cref{equivalence}}\label{sec:theproof}

The proof of \Cref{equivalence} has two main ingredients besides \Cref{poly-coefficient}. First, polynomial dependence on the clique number as in $\textup{(iv)}$ allows us to convert ordinary separators into clique-based ones by peeling large cliques. Second, the passage from $\alpha$-separators to ordinary separators uses induced bipartite subgraphs to control degeneracy on every bounded-clique slice.

We begin with the clique-peeling argument showing $\textup{(iv)}\Rightarrow\textup{(v)}$. 

\begin{theorem}
\label{peeling}
Let $\C$ be a hereditary class. Suppose that there exist constants $K\ge1$, $\beta>0$, and $s\ge0$ such that every $G\in\C$ satisfies $\mathsf{sep}(G) \leq K\omega(G)^s|V(G)|^{1-\beta}$. Then there exist constants $c\ge1$ and $\zeta > 0$ such that every $G\in\C$ has a balanced clique-based separator $\mathcal S$ satisfying 
\[
  |\mathcal S|\le c|V(G)|^{1-\zeta}.
\]
\end{theorem}

\begin{proof} By decreasing $\beta$ if necessary, we may assume that
$0<\beta\le1$. Set $\gamma:=\beta/(s+1) > 0$.
Let $G\in\C$ be a graph on $n\ge1$ vertices, and set $q:=\lceil n^\gamma\rceil$. Starting with $G$, repeatedly choose a clique of size at least $q+1$ in the current induced graph and delete all of its vertices, stopping when no such clique remains. Let the deleted cliques be $C_1,\ldots,C_k$, and let $X:=\bigcup_{i=1}^k C_i$, and $R:=G-X$.
The cliques $C_i$ are pairwise vertex-disjoint and each has at least $q+1$
vertices. Hence,
\begin{equation}\label{eq:number-peeled-cliques}
  k\le\frac{n}{q+1}\le n^{1-\gamma}.
\end{equation}
Moreover, $\omega(R)\le q$ by the stopping condition.

If $R$ is empty, then $\{C_1,\ldots,C_k\}$ is trivially a clique-based
separator of $G$ of size at most $n^{1-\gamma}$. Assume henceforth that $R$ is nonempty. Since $R$ is an induced subgraph of $G$, it belongs to $\C$. Choose a balanced separator $S_0$ of $R$ with
\[
  |S_0|
  \le K\omega(R)^s|V(R)|^{1-\beta}
  \le Kq^s n^{1-\beta}
  \le K2^s n^{1-\beta+s\gamma},
\]
where the last inequality follows from $q\le n^\gamma+1\le2n^\gamma$. Since the definition of $\gamma$ gives $1-\beta+s\gamma=1-\gamma$, we then obtain \begin{equation}\label{eq:S0-final}
  |S_0|\le K2^s n^{1-\gamma}.
\end{equation}

Now let $\mathcal S:=\{C_1,\ldots,C_k\}\cup\{\{v\}:v\in S_0\}$.
Since $S_0\subseteq V(R)=V(G)\setminus X$, the members of $\mathcal S$ are
pairwise vertex-disjoint cliques. Their union is $X\cup S_0$, and every
component of $G-(X\cup S_0)=R-S_0$ has at most $2|V(R)|/3\le2n/3$ vertices. Hence, $\mathcal S$ is a clique-based separator of $G$. By \eqref{eq:number-peeled-cliques} and \eqref{eq:S0-final},
\[
  |\mathcal S|=k+|S_0|
  \le (1+K2^s)n^{1-\gamma}.
\]
Thus the conclusion holds with $\zeta:=\gamma$ and $c:=1+K2^s$.
\end{proof}

As mentioned, to close the implication cycle in \Cref{equivalence}, we will also need to recover ordinary strongly sublinear separators from strongly sublinear $\alpha$-separators on every bounded-clique slice. The following result allows us to transfer the bounded degeneracy obtained on the bipartite members of the class to all bounded-clique members.

\begin{theorem}[\citet{KLST20}]\label{induced-bipartite} For every graph $H$, there is a constant $c_H > 0$ such that every $H$-subgraph-free graph of minimum degree at least $t$ contains an induced bipartite subgraph of minimum degree at least $c_H\log t/\log\log t$.
\end{theorem}

We can finally prove \Cref{equivalence}, which we restate for convenience. 

\equivalence*

\begin{proof}
$\textup{(i)} \Leftrightarrow \textup{(ii)}$. It immediately follows from \Cref{subgraph-separator} and the equivalence between polynomial expansion and strongly sublinear separators due to \citet[Corollary~2 and Theorem~3]{DN16}.

$\textup{(ii)} \Leftrightarrow \textup{(iii)}$. If \textup{(iii)} holds, using $\mathsf{sep}(G)\le \tw(G)+1$, we immediately obtain \textup{(ii)}. Conversely, suppose that \textup{(ii)} holds and fix $q\in\mathbb N$. Then there exist constants $c_q\ge1$ and $\varepsilon_q>0$ such that every $n$-vertex graph in $\C_{\le q}$ has a balanced separator of size at most $c_qn^{1-\varepsilon_q}$. By \cref{subgraph-separator}, the same bound holds for every subgraph of every graph in $\C_{\le q}$. Hence, $\sepnum(G)\le c_q|V(G)|^{1-\varepsilon_q}$ for every $G\in\C_{\le q}$. Applying the inequality $\tw(G)\le15\sepnum(G)$ (see \Cref{DNsep}) gives \textup{(iii)}.

$\textup{(ii)}\Rightarrow\textup{(iv)}$. Let $f\colon\mathbb N\to\mathbb R_{>0}$ and $g\colon\mathbb N\to(0,1]$ be as in the definition of strongly sublinear $(\mathsf{sep},\omega)$-boundedness. For each $q\in\mathbb{N}$, set $c_q^*:=\max\{\{1\}\cup\{f(j):1\le j\le q\}\}$, and $\varepsilon_q^*:=\min_{1\le j\le q}g(j)$. Hence, if $H\in\C_{\le q}$ and $j=\omega(H)$, then $f(j)\le c_q^*$ and $g(j)\ge\varepsilon_q^*$. It follows that $\mathsf{sep}(H) \le f(j)|V(H)|^{1-g(j)} \le c_q^*|V(H)|^{1-\varepsilon_q^*}$,
and so every fixed-clique slice has strongly sublinear separators. It is then enough to apply \Cref{poly-coefficient}.

$\textup{(iv)}\Rightarrow\textup{(v)}$. This is just \Cref{peeling}.

$\textup{(v)}\Rightarrow\textup{(vi)}$. Let $c\ge1$ and $\zeta>0$ be such that every $n$-vertex graph in $\C$ has a clique-based separator consisting of at most $cn^{1-\zeta}$ cliques. Then the weight of such a separator is at most $cn^{1-\zeta}\log(n+1)=O(n^{1-\zeta/2})$.

$\textup{(v)} \Rightarrow \textup{(vii)}$. Let $c\ge1$ and $0 < \zeta < 1$ be as in (v) and consider the nondecreasing function $h(n):=cn^{1-\zeta}$.
Then $h\left(\left\lfloor 2n/3\right\rfloor\right) \le \left(2/3\right)^{1-\zeta}h(n)$, and $(2/3)^{1-\zeta}<1$. Hence we can apply \cite[Lemma~1.1]{DMMY25}, which gives $\tin(G)=O(|V(G)|^{1-\zeta})$ for every $G\in\C$.

$\textup{(vii)} \Rightarrow \textup{(viii)}$. There exist $c\ge1$ and
$\zeta>0$ such that $\tin(G)\le c|V(G)|^{1-\zeta}$ for every $G\in\C$. Thus $G$ has a tree decomposition in which every bag $X$ satisfies $\alpha(G[X])\le c|V(G)|^{1-\zeta}$. By \Cref{balanced-bag}, one of its bags $S$ is a balanced separator of $G$.

$\textup{(vi)}\Rightarrow\textup{(viii)}$. Let $\mathcal S$ be a clique-based separator of $G \in \C$ and let $S=\bigcup_{C\in\mathcal S}C$. 
Every independent set in $G[S]$ contains at most one vertex from each
member of $\mathcal S$, and hence $\alpha(G[S])\le|\mathcal S|$.
Moreover, since $\log(|C|+1)\ge1$ for every nonempty clique $C$, we have $|\mathcal S| \le \sum_{C\in\mathcal S}\log(|C|+1),$ where the sum is exactly the weight of $\mathcal{S}$.

$\textup{(viii)}\Rightarrow\textup{(ii)}$. Choose $c\ge1$ and $\varepsilon\in(0,1]$ such that every $n$-vertex graph $G\in\C$ has a balanced separator $S$ satisfying $\alpha(G[S])\le cn^{1-\varepsilon}$.
Consider the hereditary subclass of bipartite graphs in $\C$. If $G$
belongs to this subclass and $S$ is a separator as above, then $G[S]$ is
bipartite, and hence $|S|\le2\alpha(G[S]) \le2c|V(G)|^{1-\varepsilon}$.
Thus this subclass has strongly sublinear ordinary separators. By
\Cref{degeneracy-from-sep}, choose an integer $d$ such that
every bipartite graph in $\C$ is $d$-degenerate.

We claim that, for every $q\in\mathbb N$, there exists an integer $\delta_q$ such that every graph in $\C_{\le q}$ is $(\delta_q-1)$-degenerate. For $q=1$, every graph in $\C_{\le1}$ is edgeless, hence $0$-degenerate. Fix now $q\ge2$. By \Cref{induced-bipartite}, there exists a constant $c_q$ such that every $K_{q+1}$-free graph of minimum degree at least $t$ contains an induced bipartite subgraph of minimum degree at least $c_q\log t/\log\log t$. Since $\log t/\log\log t$ tends to infinity with $t$, choose an integer $\delta_q$ such that $c_q \log \delta_q/\log\log \delta_q\ge d+1$. Suppose, to the contrary, that some $G\in\C_{\le q}$ contains an induced subgraph $H$ with $\delta(H)\ge \delta_q$. Since $\omega(H)\le q$, the graph $H$ is $K_{q+1}$-free, and hence it contains an induced bipartite subgraph $J$ with $\delta(J)\ge d+1$. Since $J\in\C$, this contradicts the choice of $d$. 

Take now $G\in\C_{\le q}$ and let $S$ be a balanced separator satisfying $\alpha(G[S])\le c|V(G)|^{1-\varepsilon}$. Since $G[S]\in\C_{\le q}$, the graph $G[S]$ is $(\delta_q-1)$-degenerate and hence $\delta_q$-colorable. Since each color class is an independent set in $G[S]$, we obtain $|S|
  \le\delta_q\alpha(G[S])
  \le \delta_q c|V(G)|^{1-\varepsilon}$.
Thus, every fixed-clique slice $\C_{\le q}$ has strongly sublinear
ordinary separators, with the same exponent $\varepsilon$, and so \textup{(ii)} holds. This completes the proof.
\end{proof}

\begin{corollary}\label{alpha-separators-poly-dgomega}
Let $\mathcal C$ be a hereditary class with strongly sublinear $\alpha$-separators. Then $\mathcal C$ is polynomially $(\dg,\omega)$-bounded.
\end{corollary}

\begin{proof}
By the proof of $\textup{(viii)}\Rightarrow\textup{(ii)}$ in
\Cref{equivalence}, for every $q\in\mathbb N$ there exists an integer $\delta_q$ such that every graph in $\mathcal C_{\le q}$ is $(\delta_q-1)$-degenerate. Hence, the function $f$ defined by $f(q):=\delta_q-1$ is a $(\dg,\omega)$-binding function for $\mathcal C$. Finally, every $(\dg,\omega)$-bounded graph class is polynomially $(\dg,\omega)$-bounded by \Cref{poly-dgnomega}.
\end{proof}

Using the clique-peeling argument from \Cref{peeling}, we next show that the strongly sublinear tree-independence conclusion of \Cref{equivalence} can be made algorithmic: For every hereditary class satisfying the equivalent conditions therein, a tree decomposition with strongly sublinear independence number can be constructed in subexponential time.  

\algo*

\begin{proof}
Choose constants $a\ge1$ and $\delta>0$ such that $\tin(H)\le a|V(H)|^{1-\delta}$ for every $H\in\mathcal C$. By \Cref{alpha-separators-poly-dgomega} and \Cref{equivalence}, after increasing constants if necessary,
there exist $A\ge1$ and an integer $s\ge0$ such that, for every $H\in\mathcal C$,  
\[
  \chi(H)\le \dg(H)+1\le A\omega(H)^s.
\] 
By decreasing $\delta$ if necessary, we may assume that $0<\delta\le1/3$. Finally, set $\gamma:=\delta/(s+1)$.

Let $G\in\mathcal C$ be a graph on $n$ vertices and set $q:=\lceil n^\gamma\rceil$. Starting with $G$, repeatedly find and delete a clique of
size $q+1$, stopping when no such clique remains. Let $C_1,\ldots,C_k$ be the deleted cliques, $X:=\bigcup_{i=1}^k C_i$, and $R:=G-X$.
Then $\omega(R)\le q$ and 
\[
   k\le\frac{n}{q+1}\le n^{1-\gamma}.
\]
Since $\mathcal C$ is hereditary, $\tin(R)\le a|V(R)|^{1-\delta} \le an^{1-\delta}$, while $\chi(R)\le A\omega(R)^s\le Aq^s$. Moreover, for every graph $H$, we have $\tw(H)+1\le\chi(H)\cdot \tin(H)$ (see, e.g., \cite{DMS24}). 
Consequently,
\[
   \tw(R)+1
   \le Aa q^s n^{1-\delta}
   \le Aa2^s n^{1-\delta+s\gamma}
   = Aa2^s n^{1-\gamma}.
\]

\citet{Kor23} showed that there exists an algorithm that, given an $n$-vertex graph $G$ and an integer $k$, in time $2^{O(k)}n$ either outputs a tree decomposition of width at most $2k+1$, or concludes that $\tw(G)>k$. Applying this, we can construct a tree decomposition $\mathcal T=(T,\{B_t\}_{t\in V(T)})$ of $R$ of width $O(n^{1-\gamma})$ in time
$2^{O(n^{1-\gamma})}$. We may assume that $\mathcal T$ has at most $n$ bags. Adding the set $X$ to every bag gives a tree decomposition $(T,\{B_t\cup X\}_{t\in V(T)})$ of $G$. Since the sets $C_1,\ldots,C_k$ are cliques, $\alpha(G[X])\le k\le n^{1-\gamma}$. Therefore, for every $t\in V(T)$,
\[
   \alpha(G[B_t\cup X])
   \le |B_t|+\alpha(G[X])
   =O(n^{1-\gamma}).
\]

Finally, finding the $(q+1)$-cliques by exhaustive search takes $2^{O(n^\gamma\log n)}$ time. Since $\gamma\le1/3$, this is $2^{O(n^{1-\gamma})}$. Hence, the total running time is $2^{O(n^{1-\gamma})}$.
\end{proof}

We conclude this section with an example showing that there is no analogue of \Cref{equivalence} for the polylogarithmic regime. 

\begin{proposition}\label{polylog}
There exist a hereditary class $\mathcal C$ of bounded tree-independence number and a constant $c>0$ such that, for infinitely many $n$, some $n$-vertex graph $G\in\mathcal C$ has the property that every clique-based separator of $G$ contains at least $c\sqrt{n/\log n}$ cliques.
\end{proposition}

\begin{proof}
Let $\mathcal C$ be the class of complements of triangle-free graphs. Since every $G\in\mathcal C$ satisfies $\alpha(G)\le2$, we have $\tin(G)\le2$.

For infinitely many $n$, there exists an $n$-vertex triangle-free graph $H$ with $\alpha(H)=O(\sqrt{n\log n})$ \cite{Kim95}. Set $G:=\overline H$. We claim that every balanced separator $S$ of $G$ has size $\Omega(n)$. Clearly, we may assume that $S \neq V(G)$. If $G-S$ is connected, then the balance condition gives $|S|\ge n/3$. Otherwise, $G-S$ has two components, since three components would yield a triangle in $H$. Moreover, each such component is an independent set in $H$ (if two vertices of one component were adjacent in $H$, together with any vertex of the other component they would form a triangle). Hence, $|V(G)\setminus S|\le2\alpha(H)=o(n)$, and so again $|S|=\Omega(n)$.

Since every clique in $G$ has size $O(\sqrt{n\log n})$, every clique-based separator $\mathcal S$ of $G$ must then satisfy $|\mathcal S| = \Omega(n/\sqrt{n\log n})$.
\end{proof}

\section{Fractional tree-$\alpha$-fragility and $\alpha$-separators}
\label{sec:frac-tree-alpha-separators}

\citet{GMY24} asked whether every graph class of bounded layered tree-independence number has sublinear-weight clique-based separators. In this section, we show that something stronger holds: Every fractionally $\tin$-fragile class has sublinear-weight clique-based separators. 

The proof relies on the following $\alpha$-analogue of a general connection between fractional fragility and sublinear separators observed by \citet[Lemma~14]{Dvo16}.

\begin{theorem}
\label{fractional-treealpha-alpha-separator}
Let $\mathcal C$ be a fractionally $\tin$-fragile class with fragility function $f$. Then, for every integer $r\ge2$ and every $n$-vertex graph $G\in\mathcal C$, there exists a balanced separator $S$ of $G$ such that
\[
  \alpha(G[S])\le \frac{n}{r}+f(r).
\]
\end{theorem}

\begin{proof}
Fix an integer $r\ge2$ and let $\mathcal A$ be a $(1-1/r)$-general cover of an $n$-vertex graph $G\in\mathcal{C}$ of $\tin$-width at most $f(r)$.
Thus, for every $v\in V(G)$, the vertex $v$ belongs to at least
$(1-1/r)|\mathcal A|$ members of $\mathcal A$. Equivalently, $v$ is
omitted from at most $|\mathcal A|/r$ members of $\mathcal A$.
Consequently,
\[
  \sum_{A\in\mathcal A}|V(G)\setminus A|
  = \sum_{v\in V(G)} |\{A\in\mathcal A:v\notin A\}|
  \le \frac{n}{r}|\mathcal A|.
\]
Hence, there exists $A\in\mathcal A$ such that $|V(G)\setminus A|\le n/r$. Moreover, by assumption, $\tin(G[A])\le f(r)$. Choose a tree decomposition of $G[A]$ whose independence number is at most $f(r)$. By \Cref{balanced-bag}, this tree decomposition has a bag $B$ such that every connected component of $G[A]-B$ has at most $|A|/2$ vertices. Set $S:=(V(G)\setminus A)\cup B$. Since $G-S=G[A]-B$, every connected component of $G-S$ has at most $|A|/2\le n/2$ vertices. Thus, $S$ is a balanced separator of $G$. Moreover,
\[
  \alpha(G[S])
  \le |V(G)\setminus A|+\alpha(G[B])
  \le \frac nr+f(r).
\]
This proves the result.
\end{proof}

As an immediate consequence, a polynomial fragility rate yields a strongly sublinear bound.

\begin{corollary}
\label{poly-fragility-alpha-separators}
Let $\mathcal C$ be a fractionally tree-$\alpha$-fragile class. Suppose
that its fragility function $f$ satisfies $f(r)\le c r^a$ for some constants $c\ge1$ and $a\ge0$ and every $r\in\mathbb N$. Then there exists a constant $C = C_{c,a}$ such that every $n$-vertex graph $G\in\mathcal C$ satisfies
\[
  \sepalpha(G)\le C n^{a/(a+1)}.
\]
\end{corollary}

\begin{proof}
Let $G\in\mathcal C$ be a graph on $n\ge2$ vertices and set $r:=\left\lceil n^{1/(a+1)}\right\rceil$. Since $n\ge2$, we have $r\ge2$, and hence
\Cref{fractional-treealpha-alpha-separator} gives
\[
  \sepalpha(G)
  \le \frac nr+f(r)
  \le \frac nr+c r^a.
\]
Since $r\ge n^{1/(a+1)}$, we have $n/r\le n^{a/(a+1)}$, whereas $r
  \le n^{1/(a+1)}+1
  \le 2n^{1/(a+1)}$,
and so
\[
  \sepalpha(G)
  \le (1+c2^a)n^{a/(a+1)}.
\]
This completes the proof.
\end{proof}

We can finally prove the desired statement. For positive integers $p$ and $q$, the \textit{Ramsey number} $R(p,q)$ is the smallest integer $n_0$ such that every graph with at least $n_0$ vertices contains either a clique of size $p$ or an independent set of size $q$.

\begin{corollary}\label{fractional-treealpha-sublinear-clique-weight}
Every fractionally $\tin$-fragile class has sublinear-weight clique-based separators. More precisely, if $\mathcal C$ is fractionally $\tin$-fragile then, for every $c>0$, there exists $n_0$ such that every $n$-vertex graph $G\in\mathcal C$ with $n\ge n_0$ has a clique-based separator of weight at most $cn$.
\end{corollary}

\begin{proof}
Let $f$ be a fragility function for $\mathcal C$, and fix $c>0$. Choose an integer $r\ge2$ such that $1/r\le c/3$. Following the proof of \Cref{fractional-treealpha-alpha-separator}, for every $n$-vertex graph $G\in\mathcal C$ there exist a set $A\subseteq V(G)$ and a set $B\subseteq A$
such that $|V(G)\setminus A|\le n/r$, $\alpha(G[B])\le f(r)$, and $(V(G)\setminus A)\cup B$ is a balanced separator of $G$.

Choose now an integer $\ell\ge2$ such that $\log(\ell+1)/\ell\le c/3$.
We partition $B$ into cliques as follows. By Ramsey's theorem, as long as the current induced subgraph has at least $R(\ell,f(r)+1)$ vertices, it contains a clique of size $\ell$. Remove such an $\ell$-vertex clique and continue. When the process stops, the set $B_0$ of remaining vertices satisfies $|B_0|<R(\ell,f(r)+1)$. Let $C_1,\ldots,C_m$ be the removed $\ell$-vertex cliques, and let $\mathcal S :=\{C_1,\ldots,C_m\} \cup\{\{v\}:v\in (V(G)\setminus A)\cup B_0\}$. The members of $\mathcal S$ are pairwise vertex-disjoint cliques and their union is $(V(G)\setminus A)\cup B$, which is a balanced separator of $G$. Therefore, $\mathcal S$ is a clique-based separator. Its weight satisfies
\begin{equation*}
  w(\mathcal S)
  \le m\log(\ell+1)+|V(G)\setminus A|+|B_0|
  \le \frac{|B|}{\ell}\log(\ell+1) +\frac nr+R(\ell,f(r)+1)
  \le \frac{2c}{3}n+R(\ell,f(r)+1).
\end{equation*}
Since for all sufficiently large $n$ it holds $R(\ell,f(r)+1)\le cn/3$, we immediately obtain that $w(\mathcal S)\le cn$.
\end{proof}

\section{Bounded-clique layered treewidth and layered tree-independence number}\label{sec:layered-tree-alpha} 

In this section we show that there exists a hereditary class whose every bounded-clique slice has bounded layered treewidth and yet the whole class has unbounded layered tree-independence number. The construction is based on the following result of \citet{CT25} and the operation of coning. For a graph $G$, the \emph{cone over $G$} is the graph $G^+$ obtained from $G$ by adding a universal vertex adjacent to all vertices of $G$.  

\begin{theorem}[\citet{CT25}]\label{CTclass} There exists a hereditary $(\tw,\omega)$-bounded class with unbounded tree-independence number.  
\end{theorem}

\begin{theorem}\label{cone-layered-counterexample} There exists a hereditary class $\mathcal C$ with the following properties. \begin{enumerate} \item For every integer $r\geq 1$, the class $\mathcal C_r=\{H\in\mathcal C:\omega(H)\le r\}$ has bounded layered treewidth. 
\item The class $\mathcal C$ has unbounded layered tree-independence number. \end{enumerate} \end{theorem} 

\begin{proof} Let $\mathcal D$ be a hereditary $(\tw,\omega)$-bounded class with unbounded tree-independence number from \Cref{CTclass}. Choose a nondecreasing function $f$ such that $\tw(G)\le f(\omega(G))$ for every $G\in\mathcal D$. Let $\mathcal C$ be the hereditary closure of the class $\{G^+:G\in\mathcal D\}$. Observe that, if $H$ is an induced subgraph of $G^+$, with $G\in\mathcal D$, then either $H$ omits the universal vertex, in which case $H$ is an induced subgraph of $G$ and hence belongs to $\mathcal D$, or $H$ contains the universal vertex, in which case $H=J^+$ for an induced subgraph $J$ of $G$, and hence for some $J\in\mathcal D$. Thus every graph in $\mathcal C$ is either a graph in $\mathcal D$ or the cone over a graph in $\mathcal D$. Moreover, $\omega(G^+)=\omega(G)+1$ and $\tw(G^+)\le \tw(G)+1$, for every graph $G$. Consequently, if $H\in\mathcal C$ and $\omega(H)\le r$, then $\tw(H)\le f(r)+1$. Hence $\mathcal C$ is $(\tw,\omega)$-bounded. Since $\ltw(H)\le \tw(H)+1$, it follows that $\ltw(H)\le f(r)+2$ whenever $H\in\mathcal C$ and $\omega(H)\le r$. This proves the first assertion. 

It remains to show that $\mathcal C$ has unbounded layered tree-independence number. We first claim that, for every nonempty graph $G$, we have $\tin(G^+)=\tin(G)$. Since $G$ is an induced subgraph of $G^+$, we immediately obtain $\tin(G)\le \tin(G^+)$. Conversely, take a tree decomposition of $G$ with independence number $\tin(G)$ and add the universal vertex to every bag. This gives a tree decomposition of $G^+$. Adding a vertex that is adjacent to every old vertex does not increase the independence number of any nonempty bag; empty bags, if present, contribute independence number at most $1\le\tin(G)$. Thus, $\tin(G^+)\le\tin(G)$. 

We next claim that $\tin(G^+)\le 3\,\ltreealpha(G^+)$, for every graph $G$. Let $u$ be the universal vertex of $G^+$, and take a tree decomposition $(T,\{B_t\}_{t\in V(T)})$ and a layering  $(V_i)_{i\in\mathbb Z}$ witnessing $\ltreealpha(G^+)=k$. Thus, $\alpha(G^+[B_t\cap V_i])\le k$ for every $t\in V(T)$ and $i\in\mathbb Z$. Suppose that $u\in V_j$. Since $u$ is adjacent to every vertex of $G^+$, we obtain $V(G^+)\subseteq V_{j-1}\cup V_j\cup V_{j+1}$. Now fix an arbitrary bag $B_t$ and let $S$ be an independent set in $G^+[B_t]$. We have \[ |S| =\sum_{i=j-1}^{j+1}|S\cap V_i| \le \sum_{i=j-1}^{j+1} \alpha\bigl(G^+[B_t\cap V_i]\bigr) \le 3k. \] Therefore, every bag of this tree decomposition has independence number at most $3k$, and hence $\tin(G^+)\le 3k$.  

Combining the previous two claims, we conclude that $\ltreealpha(G^+) \ge \tin(G^+)/3 = \tin(G)/3$. Since $\mathcal D$ has unbounded tree-independence number, the class of cones $G^+$ with $G\in\mathcal D$ has unbounded layered tree-independence number, and hence $\mathcal C$ has unbounded layered tree-independence number. 
\end{proof} 

\begin{remark} The preceding construction turns an obstruction to bounded tree-independence number into an obstruction to bounded layered tree-independence number, but it does not by itself yield an obstruction to fractional $\tin$-fragility. Indeed, let $\mathcal D$ be a hereditary class and let $\mathcal C$ be the hereditary closure of $\{G^+:G\in\mathcal D\}$ as above. We claim that $\mathcal C$ is fractionally $\tin$-fragile if and only if $\mathcal D$ is.

One direction follows immediately from $\mathcal D\subseteq\mathcal C$. Conversely, suppose that $\mathcal D$ is fractionally $\tin$-fragile, with fragility function $f$, and let $G\in\mathcal D$. For $r\in\mathbb N$, let $\mathcal A$ be a $(1-1/r)$-general cover of $G$ of $\tin$-width at most $f(r)$. If $u$ is the universal vertex of $G^+$, define $\mathcal A^+:=\{A\cup\{u\}:A\in\mathcal A\}$, with multiplicities inherited from $\mathcal A$. Every vertex of $G$ belongs to the same proportion of members of $\mathcal A^+$ as it does of $\mathcal A$, while $u$ belongs to every member. Hence, $\mathcal A^+$ is a $(1-1/r)$-general cover of $G^+$. Moreover, for every $A\in\mathcal A$,
\[
\tin(G^+[A\cup\{u\}])
 =\tin((G[A])^+)
 \le \max\{1,\tin(G[A])\}
 \le \max\{1,f(r)\},
\]
and so $\mathcal C$ is fractionally $\tin$-fragile.

Consequently, coning separates bounded layered tree-independence number from bounded layered treewidth on the bounded-clique slices, but it does not turn the Chudnovsky--Trotignon construction into a counterexample to the corresponding implication for fractional $\tin$-fragility.
\end{remark}

\section{Sublinear separators and $\chi$-boundedness}\label{sec:counter}

\citet[Question~7.2]{DMMY25} asked whether every graph class admitting sublinear-weight clique-based separators has bounded $\alpha$-degeneracy. A negative answer to this question can be obtained from a construction of \citet[Theorem~1]{MS18}. Indeed, they construct subgraphs of the hypercube with arbitrarily large average degree such that every $t$-vertex
subgraph has a balanced separator of size $O(t(\log\log t)^2/\log t)$. Taking the monotone closure of these graphs gives a monotone bipartite class with sublinear separators and unbounded degeneracy. Since the neighborhood of every vertex in a bipartite graph is independent, $\alpha$-degeneracy and degeneracy coincide on bipartite graphs. Thus, this construction gives a negative answer to the question above.

We show that a considerably stronger phenomenon occurs: Sublinear separators
do not imply $\chi$-boundedness even for monotone triangle-free classes. In particular, the following result also gives a negative answer to \cite[Question~7.2]{DMMY25}.

\begin{theorem}\label{not-chi-bound}
There exists a monotone class $\mathcal{C}$ consisting of
triangle-free graphs with unbounded chromatic number and sublinear separators.
\end{theorem}

The proof of \Cref{not-chi-bound} is inspired by the construction of \citet{MS18}, who randomly sparsify the binary hypercube and combine large girth with its metric structure to exclude expanding subgraphs. To obtain unbounded chromatic number, we replace the hypercube by the \emph{Hamming graph} $H(q,d)$, the graph on $[q]^d$ in which two vertices are adjacent if and only if they differ in exactly one coordinate. The same non-expansion mechanism as in \cite{MS18} survives, while the cliques obtained by fixing all but one coordinate allow us to force small independence number.
We first show in \Cref{girth-bound} that every bounded-degree edge-expander in $H(q,d)$ containing a cycle has girth $O(\log d)$, and then construct in \Cref{random} bounded-degree subgraphs of $H(q,d)$ with arbitrarily large chromatic number and girth much larger than this. The resulting failure of expansion is then converted into sublinear balanced separators by an iterative peeling argument.

For $\varepsilon>0$, a graph $G$ is an \emph{$\varepsilon$-edge-expander} if every nonempty set $A\subseteq V(G)$ with $|A|\le |V(G)|/2$ satisfies $|\partial_G(A)|\ge \varepsilon |A|$, where $\partial_G(A)$ denotes the set of edges of $G$ with exactly one endpoint in $A$.

\begin{lemma}\label{girth-bound}
Fix $q\ge2$, $M\ge1$, and $\varepsilon>0$. There is a constant $C=C(q,M,\varepsilon)$ such
that every $\varepsilon$-edge-expander subgraph $J$ of $H(q,d)$ that contains a cycle and
has maximum degree at most $M$ satisfies $\girth(J)\le C\log(d+1)$.
\end{lemma}

\begin{proof}
Let $J$ be as in the statement and set $t:=|V(J)|$ and $H:=H(q,d)$. We first bound the girth of $J$ in terms of $t$. Since $\varepsilon>0$, the graph $J$ is connected. Let $B_s(v) := \{u \in V(J): d_{J}(u,v) \leq s\}$ be the radius-$s$ ball in $J$ centered at $v$. As long as $|B_s(v)|\le t/2$, the expansion property gives at least $\varepsilon |B_s(v)|$ edges from $B_s(v)$ to its complement in $V(J)$. Since every vertex of $J$ has degree at most $M$, these edges have at least $\varepsilon|B_s(v)|/M$ distinct endpoints outside $B_s(v)$. Hence, 
\[
  |B_{s+1}(v)|
  \ge \left(1+\frac{\varepsilon}{M}\right)|B_s(v)|,
\] 
for every $v \in V(J)$. Now set $r:=\left\lceil
\log t/\log(1+\varepsilon/M)\right\rceil$. If $|B_r(v)|\le t/2$, then the previous inequality can be iterated for $s=0,\ldots,r-1$, yielding $|B_r(v)|
\ge (1+\varepsilon/M)^r \ge t$, a contradiction. Hence, every ball of radius $r$ contains more than $t/2$ vertices. Any two such balls therefore intersect, and so $\diam(J)\le 2r=O_{M,\varepsilon}(\log t)$, from which $\girth(J) \leq 2\diam(J) +1 = O_{M,\varepsilon}(\log t)$.

We finally show that $\log t = O_{q,M,\varepsilon}(\log(d+1))$. For each coordinate $i\in[d]$, choose a value $a_i\in[q]$ which occurs most frequently in the $i$-th coordinate among the vertices of $J$, and set $a:=(a_1,\ldots,a_d)$. For $b\in[q]$, let $A_{i,b} = \{v\in V(J) : v_i = b\}$ be the set of vertices whose $i$-th coordinate is $b$. If $b\neq a_i$, then $|A_{i,b}|\le t/2$. Hence, the expansion property gives $|\partial_J(A_{i,b})|\ge \varepsilon |A_{i,b}|$. Therefore, double counting the number of pairs $(v,i)$ such that $v_i\neq a_i$, we obtain
\[
  \varepsilon\sum_{v\in V(J)} d_H(v,a)
  = \varepsilon\sum_{v\in V(J)}|\{i \in [d]: v_i \neq a_i\}|
  = \varepsilon\sum_{i=1}^d\sum_{b\neq a_i}|A_{i,b}|
  \le \sum_{i=1}^d\sum_{b\neq a_i}|\partial_J(A_{i,b})|.
\]
Observe now that every edge of $J$ is counted in the rightmost sum at most twice. Thus,
\[
  \varepsilon\sum_{v\in V(J)} d_H(v,a)
  \le 2|E(J)|
  \le Mt.
\]
It follows that the average distance from $a$ is at most $M/\varepsilon$. Consequently, at least $t/2$ vertices of $J$ have distance at most $\left\lceil 2M/\varepsilon\right\rceil$ from $a$. Set $\ell:= \left\lceil 2M/\varepsilon\right\rceil$. Since a ball of radius $s$ in $H$ has at most
$\sum_{j=0}^{s}\tbinom{d}{j}(q-1)^j$ vertices, we obtain
\[
  \frac t2
  \le \sum_{j=0}^{\ell}\binom{d}{j}(q-1)^j
  = O_{q,M,\varepsilon}(d^\ell),
\]
which gives $\log t=O_{q,M,\varepsilon}(\log(d+1))$.
\end{proof}

The second lemma is a Hamming-graph version of the classical probabilistic construction of high-girth, high-chromatic graphs of \citet{Erd59}, combined
with the sparsification idea of \citet{MS18}. We retain every edge of $H(q,d)$ independently with probability $q/d$. For each coordinate, fixing all the others partitions the vertex set into cliques of size $q$. This implies, via a union bound, that the resulting graph has small independence number. At the same time, only few vertices have large degree.
Moreover, a cycle of length $\ell$ uses at most $\ell/2$ coordinates, which yields a sufficiently small expected number of short cycles. Deleting the
high-degree vertices and one vertex from each short cycle leaves a bounded-degree graph of large girth and chromatic number larger than $q/4$.

\begin{lemma}\label{random}
For every integer $k\ge1$ and every sufficiently large $d$, there is a
subgraph $G$ of $H(4k,d)$ with
\[
 \chi(G)>k,\qquad \Delta(G)\le8(4k)^2,
 \qquad \girth(G)>\lfloor d^{1/4}\rfloor.
\]
\end{lemma}

\begin{proof}
Set $q:=4k$, $n:=q^d$, and $g:=\lfloor d^{1/4}\rfloor$. Let $R$ be the random spanning subgraph of $H(q,d)$ obtained by keeping each edge independently with probability $p:=q/d$. Throughout the proof, $q$ is fixed and $d\to\infty$.

We first show that, with probability tending to one, $\alpha(R)<2n/q$. Set $s:=2n/q$, and fix a set $S\subseteq V(H(q,d))$ of size $s$. For a fixed coordinate $i\in[d]$, group together two vertices whenever they agree in every coordinate other than $i$. This partitions $V(H(q,d))$ into $q^{d-1}=n/q$ sets of size $q$, each of which is a clique. Let $x_1,\ldots,x_{n/q}$ be the sizes of the intersections of these sets
with $S$. Clearly, $\sum_{j=1}^{n/q}x_j=s$. The Cauchy-Schwarz inequality then gives
\[
  \sum_{j=1}^{n/q}\binom{x_j}{2}
  = \frac12\left(\sum_{j=1}^{n/q}x_j^2-s\right)
  \ge \frac12\left(\frac{s^2}{n/q}-s\right)
  = \frac{n}{q}.
\]
Thus, for every coordinate $i$, the set $S$ contains at least $n/q$ edges whose endpoints differ in coordinate $i$. Since every edge of $H(q,d)$ has endpoints differing in exactly one coordinate, the edges obtained from distinct coordinates are disjoint. Consequently, $S$ induces at least $dn/q$ edges in $H(q,d)$. Therefore, 
\[
  \mathbb{P}(S\text{ is independent in }R)
  \le (1-p)^{dn/q}.
\] 
Taking a union bound over all $s$-vertex subsets of $V(H(q,d))$, we
obtain
\[
  \mathbb{P}\bigl(\alpha(R)\ge s\bigr)
  \le \binom{n}{s}(1-p)^{dn/q}
  \le 2^n(1-p)^{dn/q}
  \le 2^n\exp\left(-\frac{pdn}{q}\right)
  = \left(\frac{2}{e}\right)^n
  = o(1).
\]

We next control the maximum degree. Every vertex of $H(q,d)$ has degree $d(q-1)$. Hence, for every vertex $v$,
\[
  \mathbb E[d_R(v)]
  = pd(q-1)
  = q(q-1)
  < q^2.
\]
By Markov's inequality,
\[
  \mathbb{P}(d_R(v)>8q^2)
  \le \frac{\mathbb E[d_R(v)]}{8q^2}
  < \frac18.
\]
Hence, denoting by $X$ the number of vertices of $R$ whose degree is greater than $8q^2$, linearity of expectation gives $\mathbb E[X]<n/8$.

We now show that $R$ contains few short cycles. For $\ell\ge3$, let $C_\ell$ denote the number of cycles of length $\ell$ in $R$. We first bound the number of cycles of length $\ell$ in $H(q,d)$. Consider such a cycle, and let $j$ be the number of coordinates that change along its edges. Each of these $j$ coordinates must change at least twice: indeed, after traversing the cycle, every coordinate must return to its initial value. Hence, $j\le \left\lfloor\ell/2\right\rfloor$. Now, there are at most $\binom{d}{j}$ ways to choose these coordinates and $n$ ways to choose a starting vertex. Starting from this vertex, each of the $\ell$ successive edges is determined by choosing one of the $j$ coordinates and changing its value to one of the other $q-1$ values. Thus, for fixed $j$, the number of possible sequences is at most $n\binom{d}{j}\bigl(j(q-1)\bigr)^\ell$. Consequently, the number of cycles of length $\ell$ in $H(q,d)$ is at most
\[
  n\sum_{j=1}^{\lfloor\ell/2\rfloor}\binom{d}{j}(j(q-1))^\ell 
  \leq n\ell\cdot d^{\ell/2}(\ell q)^\ell.
\]
Combining the above with the fact that a fixed cycle of length $\ell$ is present in $R$ with probability $p^\ell$ and using $p = q/d$, we obtain 
\[
  \mathbb{E}[C_\ell]
  \le n\ell \cdot d^{\ell/2}(\ell q)^\ell \cdot p^\ell 
  = n\ell\left(\frac{q^2\ell}{\sqrt d}\right)^\ell.
\]
Let $Y:=\sum_{\ell=3}^{g} C_\ell$ be the total number of cycles of length at most $g$ in $R$. By linearity of expectation and recalling that $g=\lfloor d^{1/4}\rfloor$, we obtain 
\[
  \mathbb{E}[Y]
  \le n\sum_{\ell=3}^{g}\ell\left(\frac{q^2\ell}{\sqrt d}\right)^\ell
  \le n\sum_{\ell=3}^{g}\ell\left(\frac{q^2}{d^{1/4}}\right)^\ell
  \le n\sum_{\ell=3}^{\infty}\ell\left(\frac{q^2}{d^{1/4}}\right)^\ell
  = o(n),
\]
as $q$ is fixed and $q^2/d^{1/4}\to0$ as $d\to\infty$.

Recall now that $X$ denotes the number of vertices of $R$ whose degree is greater than $8q^2$, and that $\mathbb{E}[X]<n/8$. Combining this with the previous paragraph gives $\mathbb{E}[X+Y] \le n/8+o(n)$. By Markov's inequality,
\[
  \mathbb{P}(X+Y>n/2)
  \le \frac{\mathbb{E}[X+Y]}{n/2}
  \le \frac14+o(1).
\]
On the other hand, we have already shown that
$\mathbb{P}(\alpha(R)\ge2n/q)=o(1)$. Thus, for all sufficiently large $d$, there is a realization of $R$ for which simultaneously $\alpha(R)<2n/q$ and $X+Y\le n/2$. Fix such a realization. Let $B$ be the set of vertices whose degree in $R$ is greater than $8q^2$, so $|B|=X$. Next, for every cycle of $R$ of length at most $g$, choose one vertex on the cycle, and let $Z$ be the set of all vertices chosen in this way. Since there are $Y$ such cycles, $|Z|\le Y$. Now let $G:=R-(B\cup Z)$. Since $|B\cup Z| \le X+Y \le n/2$, we have $|V(G)|\ge n/2$. Moreover, by construction, $
\Delta(G)\le8q^2$ and  $\girth(G)>g$. Finally, since $\alpha(G)\le\alpha(R)<2n/q$, it follows that
\[
  \chi(G)
  \ge \frac{|V(G)|}{\alpha(G)}
  > \frac{n/2}{2n/q}
  = \frac{q}{4}
  = k.
\]
This completes the proof. 
\end{proof}

Combining \Cref{girth-bound,random}, we can finally prove \Cref{not-chi-bound}. 

\begin{proof}[Proof of \Cref{not-chi-bound}]
For each $k\ge1$, apply \Cref{girth-bound} with $q=4k$, $M=8q^2$, and $\varepsilon=1/k$ to obtain a constant $C = C(q,M,1/k)$. Choose $d_k$ sufficiently large that \Cref{random} applies and $\lfloor d_k^{1/4}\rfloor>\max\{3,C\log(d_k+1)\}$.
Let $G_k$ be given by \Cref{random}. By \Cref{girth-bound}, no subgraph of $G_k$ containing a cycle is a $(1/k)$-edge-expander. We show the following.

\begin{nitem}\label{sepG_k} Every $n$-vertex subgraph $H$ of $G_k$ has a balanced separator $S$ of size at most $n/k+1$. 
\end{nitem}

\begin{claimproof}[Proof of \eqref{sepG_k}] We construct the separator iteratively. Set $J_0:=H$. Suppose that $J_i$ has already been defined. If $|V(J_i)|\le 2n/3$ or $J_i$ is a forest, stop. Otherwise, $J_i$ contains a cycle. Since $J_i$ is a subgraph of $G_k$, the preceding observation implies that $J_i$ is not a $(1/k)$-edge-expander. Hence, there exists a nonempty set $A_i\subseteq V(J_i)$ such that $|A_i|\le |V(J_i)|/2$ and $|\partial_{J_i}(A_i)|<|A_i|/k$. Let $S_i$ be the set of vertices of $A_i$ incident with an edge of $\partial_{J_i}(A_i)$. Since each edge of $\partial_{J_i}(A_i)$ has exactly one endpoint in $A_i$, we have $|S_i|\le |\partial_{J_i}(A_i)| <|A_i|/k$. We then set $J_{i+1}:=J_i-A_i$. 

The vertices of $S_i$ will belong to the final separator. By the definition of $S_i$, there is no edge of $J_i-S_i$ between $A_i\setminus S_i$ and $V(J_{i+1})$. Thus, after deleting the separator vertices, every component contained in $A_i\setminus S_i$ is separated from all vertices considered at subsequent steps. Moreover, $|A_i\setminus S_i|
\le |A_i|
\le |V(J_i)|/2
\le n/2$.
Since at least one vertex is removed from $J_i$ at every step, the process eventually terminates. Let $J_m$ be the graph obtained at termination and let $S':=\bigcup_{i=0}^{m-1}S_i$. Since the sets $A_0,\ldots,A_{m-1}$ are pairwise disjoint, we have 
\[
  |S'|
  \le \sum_{i=0}^{m-1}|S_i|
  \leq \frac1k\sum_{i=0}^{m-1}|A_i|
  \le \frac nk.
\]

We now complete the separator. If $|V(J_m)|\le 2n/3$, set $S:=S'$. Otherwise, by the stopping rule, $J_m$ is a forest. If every component of $J_m$ has at most $2n/3$ vertices, again set $S:=S'$. Otherwise, there is a unique component $T$ of $J_m$ with more than $2n/3$ vertices. Choose a centroid $v$ of the tree $T$ and set $S:=S'\cup\{v\}$. Every component of $T-v$ has at most $|V(T)|/2\le n/2$ vertices.

Clearly, $|S|\le n/k+1$. Moreover, every component separated off at one of the iterative steps has at most $n/2$ vertices. The components contained in the final graph $J_m-S$ have at most $2n/3$ vertices by construction. Hence, every component of $H-S$ has at most $2n/3$ vertices. 
\end{claimproof}

Let $\mathcal C$ be the class of all graphs isomorphic to subgraphs of some $G_k$, for $k \in\mathbb{N}$. By definition, $\mathcal C$ is monotone. Moreover, every graph in $\mathcal C$ is triangle-free, while $\chi(G_k)>k$ for every $k$, so $\mathcal C$ has unbounded chromatic number. It remains to verify that $\mathcal{C}$ has sublinear separators. Fix an integer $K\ge2$, and let $N_K:=\max_{1\le k<K}|V(G_k)|$. If $H\in\mathcal C$ has $n>N_K$ vertices, then $H$ cannot be a subgraph of any $G_k$ with $k<K$. Hence, $H$ is a subgraph of some $G_k$ with $k\ge K$. By~\eqref{sepG_k},
\[
  \mathsf{sep}(H)
  \le \frac nk+1
  \le \frac nK+1.
\]
Let $\varepsilon>0$ and choose $K$ so large that $1/K<\varepsilon/2$. Then, for every $n>\max\left\{N_K,2/\varepsilon\right\}$ and every $n$-vertex graph $H\in\mathcal C$, we have
\[
  \frac{\mathsf{sep}(H)}{n}
  \le \frac1K+\frac1n
  < \varepsilon.
\]
This concludes the proof.
\end{proof}

\section{Concluding remarks}

Our main result, \Cref{equivalence}, provides a strongly sublinear variant of a recent result of \citet{CESL25}. More generally, it shows that a large number of algorithmically useful properties that were introduced independently in the literature in fact describe the same hereditary graph classes. Together with \Cref{algo}, this also yields a subexponential-time algorithm for constructing tree decompositions of strongly sublinear independence number. 

We believe that our equivalence theorem can be further extended to capture another property particularly useful in designing PTASes for maximization problems, namely fractional $\tin$-fragility with polynomial fragility function. We observed in \Cref{poly-fragility-alpha-separators} that polynomial-rate fractional $\tin$-fragility implies the equivalent conditions of \Cref{equivalence}. A first step towards the converse direction would be to prove \Cref{alphaDvorak}, and we remark that an interesting testbed is the class of pseudo-disk graphs; whether it is fractionally $\tin$-fragile was already asked in \cite[Question~7]{GMY24}.   

\paragraph{AI disclosure.} ChatGPT-5.6 Sol was used as follows. It suggested the critical use of the induced sparsification lemma (\Cref{induced-sparsification}) to prove \textup{(ii)}$\Rightarrow$\textup{(iv)} in \Cref{equivalence}. It assisted in developing a number of technical details in the proof of \Cref{equivalence}. Finally, the proof of \Cref{not-chi-bound} is entirely due to ChatGPT. In all these cases the author reviewed and edited the output as needed, and takes full responsibility for the content of the paper.

\bibliographystyle{plainnat}
\bibliography{biblio}

\end{document}